\documentclass[a4paper]{amsart}

\usepackage[utf8]{inputenc}
\usepackage{amssymb}
\usepackage{amsthm,amsmath}
\usepackage{graphicx}
\usepackage{amsfonts}
\usepackage{amssymb,latexsym}
\usepackage{url}
\usepackage{amsmath}
\usepackage{paralist}
\usepackage{tikz}
\usetikzlibrary{calc}
\usepackage{enumitem}
\usepackage{color}
\usetikzlibrary{calc,decorations.pathreplacing}
\usetikzlibrary{decorations.text}
\newtheorem{theorem}{Theorem}[section]

\newtheorem{corollary}[theorem]{Corollary}
\newtheorem{proposition}[theorem]{Proposition}
\newtheorem{lemma}[theorem]{Lemma}
\theoremstyle{definition}

\newtheorem*{problem}{Problem}

\newtheorem{remark}[theorem]{Remark}

\newcommand{\nset}[1]{[{#1}]}
\newcommand{\IB}{\mathbb{B}}
\newcommand{\IN}{\mathbb{N}}

\newcommand{\IR}{\mathbb{R}}
\newcommand{\True}[1]{{#1}^{-1}(1)}  
\newcommand{\False}[1]{{#1}^{-1}(0)}  
\newcommand{\dual}[1]{{#1}^{\mathrm{d}}}
\newcommand{\comp}[1]{\overline{#1}}
\newcommand{\thr}{T}

\def\2{{\mathbb{B}}}

\DeclareMathOperator{\Pol}{cPol}
\DeclareMathOperator{\Inv}{cInv}

\DeclareMathOperator{\forbid}{forbid}

\def\aa{\mathbf{a}}
\def\bb{\mathbf{b}}

\def\ee{\mathbf{e}}

\def\xx{\mathbf{x}}
\def\yy{\mathbf{y}}
\def\zz{\mathbf{z}}
\def\uu{\mathbf{u}}
\def\vv{\mathbf{v}}
\def\ww{\mathbf{w}}
\def\00{\mathbf{0}}
\def\11{\mathbf{1}}
\def\IN{\mathbb{N}}

\newcommand{\Ptwoone}{$\Omega(1)$}
\newcommand{\PostsLattice}[1]{
  \begin{tikzpicture}[scale=#1, transform shape]
    \tikzstyle{every node} = [circle, fill=black,scale=0.5]
    \tikzstyle{every label} = [scale=2,draw=none, fill=none, label distance=-4]
    \node (P2) [label=above:$\Omega$] at (7.0   ,10.0) {};
    \node (T0) [label=above left:$T_0$]  at (7.0-1.5,10.0-.5) {};
    \node (T1) [label=above right:$T_1$] at (7.0+1.0,10.0-.5) {};
    \node (T)  at (7.0-.5,10.0-1.0) {};
    \node (M) [label=above left:$M$] at (6.76   ,9.38) {};
    \node (T0M) at (7.23-1.5, 9.51-.5) {};
    \node (T1M) at (6.76+1.0, 9.38-.5) {};
    \node (TM)  at (7.23-.5, 9.51-1.0) {};
    \node (L) [label=above right:$L$] at (7.0   , 5.2) {};
    \node (T0L) at (7.0-1.5, 5.2-.5) {};
    \node (T1L) at (7.0+1.0, 5.2-.5) {};
    \node (TL)  at (7.0-.5, 5.2-1.0) {};
    \node (AV)  at (7.0   , 3.0) {};
    \node (T0AV) at (7.0-1.5, 3.0-.5) {};
    \node (T1AV) at (7.0+1.0, 3.0-.5) {};
    \node (TAV)  at (7.0-.5, 3.0-1.0) {}; 
    
    \node (P21) [label=right:\Ptwoone] at (7.0,3.8) {}; 
    \node (S) [label=right:$S$] at (6.85,7.0) {};
    \node (SL) at (6.75,4.7) {};
    \node (SP21) at (6.65,2.5) {};
    \node (ST) at (6.5,6.5) {};
    \node (SM) [label=below left:$SM$] at (6.15,6.0) {};

    \foreach \from/\to in {
          P2/T0, P2/T1, T0/T, T1/T,
          M/T0M, M/T1M, T0M/TM, T1M/TM,
          P2/M, T0/T0M, T1/T1M, T/TM,
          L/T0L, L/T1L, T0L/TL, T1L/TL,
          P2/L, T0/T0L, T1/T1L, T/ST, ST/TL,
          AV/T0AV, AV/T1AV, T0AV/TAV, T1AV/TAV,
          L/P21, P21/AV, T0L/T0AV, T1L/T1AV, TL/TAV,
          P2/S, S/ST, S/SL, L/SL, SL/TL, SL/SP21, P21/SP21, SP21/TAV,
          ST/SM, SM/TAV}
    \draw [-] (\from) -- (\to);

    \coordinate (DiffT0xT)  at ( 1.0,-0.7);
    \coordinate (DiffT0xM)  at ( 2.0,-0.3);
    \coordinate (DiffT0xTM) at ( 3.0,-1.0);
    \node (T02) [label=left:$U_2$] at (-0.1, 8.5) {};
    \node (T02T)  at ($ (T02) + (DiffT0xT)  $) {};
    \node (T02M)  at ($ (T02) + (DiffT0xM)  $) {};
    \node (T02TM) at ($ (T02) + (DiffT0xTM) $) {};
    \node (T03) [label=left:$U_3$] at ($ (T02) - ( 0.0, 1.0) $) {};
    \node (T03T)  at ($ (T03) + (DiffT0xT)  $) {};
    \node (T03M)  at ($ (T03) + (DiffT0xM)  $) {};
    \node (T03TM) at ($ (T03) + (DiffT0xTM) $) {};
    \node (T0H) [draw=none, fill=none, scale=0.1] at ($ (T02) - ( 0.0, 1.5) $) {};
    \node (T0HT) [draw=none, fill=none, scale=0.1] at  ($ (T0H) + (DiffT0xT)  $) {};
    \node (T0HM) [draw=none, fill=none, scale=0.1] at  ($ (T0H) + (DiffT0xM)  $) {};
    \node (T0HTM) [draw=none, fill=none, scale=0.1] at ($ (T0H) + (DiffT0xTM) $) {};
    \node (T0e) [label=left:$U_{\infty}$] at ($ (T02) - ( 0.0, 2.3) $) {};
    \node (T0eT)  at ($ (T0e) + (DiffT0xT)  $) {};
    \node (T0eM)  at ($ (T0e) + (DiffT0xM)  $) {};
    \node (T0eTM) [label=right:$M_cU_{\infty}$] at ($ (T0e) + (DiffT0xTM)  $) {};
    
    \node (A) [label=below:$\Lambda$] at (.5+3.5,3.7) {};
    \node (AT1) at ($ (A) + ( 1.0,-0.5) $) {};
    \node (AT0) at ($ (A) + (-1.5,-0.5) $) {};
    \node (AT)  at ($ (A) + (-0.5,-1.0) $) {}; 

    \coordinate (DiffT1xT)  at (-1.0,-0.7);
    \coordinate (DiffT1xM)  at (-2.0,-0.3);
    \coordinate (DiffT1xTM) at (-3.0,-1.0);
    \node (T12) [label=right:$W_2$] at (13.6, 8.5) {};
    \node (T12T)  at ($ (T12) + (DiffT1xT)  $) {};
    \node (T12M)  at ($ (T12) + (DiffT1xM)  $) {};
    \node (T12TM) at ($ (T12) + (DiffT1xTM) $) {};
    \node (T13) [label=right:$W_3$] at ($ (T12) - ( 0.0, 1.0) $) {};
    \node (T13T)  at ($ (T13) + (DiffT1xT)  $) {};
    \node (T13M)  at ($ (T13) + (DiffT1xM)  $) {};
    \node (T13TM) at ($ (T13) + (DiffT1xTM) $) {};
    \node (T1H) [draw=none, fill=none, scale=0.1] at ($ (T12) - ( 0.0, 1.5) $) {};
    \node (T1HT) [draw=none, fill=none, scale=0.1] at  ($ (T1H) + (DiffT1xT)  $) {};
    \node (T1HM) [draw=none, fill=none, scale=0.1] at  ($ (T1H) + (DiffT1xM)  $) {};
    \node (T1HTM) [draw=none, fill=none, scale=0.1] at ($ (T1H) + (DiffT1xTM) $) {};
    \node (T1e) [label=right:$W_{\infty}$] at ($ (T12) - ( 0.0, 2.3) $) {};
    \node (T1eT)  at ($ (T1e) + (DiffT1xT)  $) {};
    \node (T1eM)  at ($ (T1e) + (DiffT1xM)  $) {};
    \node (T1eTM) [label=left:$M_cW_{\infty}$] at ($ (T1e) + (DiffT1xTM)  $) {};

    \node (V) [label=below:$V$] at (9.5,3.7) {};
    \node (VT0) at ($ (V) + (-1.0,-0.5) $) {};
    \node (VT1) at ($ (V) + ( 1.5,-0.5) $) {};
    \node (VT)  at ($ (V) + ( 0.5,-1.0) $) {}; 

    \foreach \from/\to in {
          T02/T02T, T02/T02M, T02T/T02TM, T02M/T02TM,
          T03/T03T, T03/T03M, T03T/T03TM, T03M/T03TM,
          T0e/T0eT, T0e/T0eM, T0eT/T0eTM, T0eM/T0eTM,
          T12/T12T, T12/T12M, T12T/T12TM, T12M/T12TM,
          T13/T13T, T13/T13M, T13T/T13TM, T13M/T13TM,
          T1e/T1eT, T1e/T1eM, T1eT/T1eTM, T1eM/T1eTM,
          T02/T03, T02T/T03T, T02M/T03M, T02TM/T03TM,
          T03/T0H, T03T/T0HT, T03M/T0HM, T03TM/T0HTM,
          T12/T13, T12T/T13T, T12M/T13M, T12TM/T13TM,
          T13/T1H, T13T/T1HT, T13M/T1HM, T13TM/T1HTM,
          T0eM/AT0, T0eTM/AT,
          T1eM/VT1, T1eTM/VT,
          A/AT0, A/AT1, AT0/AT, AT1/AT,
          V/VT0, V/VT1, VT0/VT, VT1/VT,
          A/AV, AT0/T0AV, AT1/T1AV, AT/TAV,
          V/AV, VT0/T0AV, VT1/T1AV, VT/TAV,
          T02TM/SM, T12TM/SM,
          T0/T02, T0M/T02M, T/T02T, TM/T02TM,
          T1/T12, T1M/T12M, T/T12T, TM/T12TM}
    \draw [-] (\from) -- (\to);
    \path (M) edge [out=215, in=70] (A);
    \path (T1M) edge [out=215, in=90] (AT1);
    \path (M) edge [out=325, in=110] (V);
    \path (T0M) edge [out=325, in=90] (VT0);

    \foreach \from/\to in {
          T0H/T0e, T0HT/T0eT, T0HM/T0eM, T0HTM/T0eTM,
          T1H/T1e, T1HT/T1eT, T1HM/T1eM, T1HTM/T1eTM}
    \draw [dotted] (\from) -- (\to);

    \draw [dashed,smooth] plot coordinates{
      ($(AT0)   +(-1.5,-0.2)$)
      ($ (L)    +( 0.1, 0.4)$)
      ($(VT1)   +( 1.5,-0.2)$)
      };
  \end{tikzpicture}
}

\begin{document}

\title[Classification of equational classes of threshold functions]{A complete classification of equational classes of threshold functions   included in clones}

\author{Miguel Couceiro}
\address[M. Couceiro]{LAMSADE -- CNRS \\
Universit\'e Paris-Dauphine \\
Place du Mar\'echal de Lattre de Tassigny \\
75775 Paris Cedex 16 \\
France}
\email{miguel.couceiro@dauphine.fr}

\author{Erkko Lehtonen}
\address[E. Lehtonen]{University of Luxembourg \\
Computer Science and Communications Research Unit \\
6, rue Richard Coudenhove-Kalergi \\
L--1359 Luxembourg \\
Luxembourg}
\email{erkko.lehtonen@uni.lu}

\author{Karsten Sch\"olzel}
\address[K. Sch\"olzel]{University of Luxembourg \\
Mathematics Research Unit \\
6, rue Richard Coudenhove-Kalergi \\
L--1359 Luxembourg \\
Luxembourg}
\email{karsten.schoelzel@uni.lu}


\begin{abstract}
 The class of threshold functions is known to be characterizable by functional equations or, equivalently, by pairs of relations, which are
 called relational constraints. It was shown by Hellerstein that this class cannot be characterized by a finite number of such objects.
In this paper, we investigate classes of threshold functions which arise as intersections of the class of all threshold functions 
with clones of Boolean functions, and provide a complete classification of such intersections in respect to whether 
they have finite characterizations.
Moreover, we provide a characterizing set of relational constraints for each class of threshold functions arising in this way.

\end{abstract}

%

\maketitle


\section{Introduction and preliminaries}

\subsection{Introduction}

Two approaches to characterize properties of Boolean functions have been considered recently:
one in terms of functional equations \cite{EFHH}, another in terms of relational constraints \cite{Pippenger}.
 As it turns out, these two approaches have the same expressive power in the sense that they characterize 
 the same properties (classes) of Boolean functions, which can be described 
 as initial segments of the so-called ``minor'' relation between functions: for two functions $f$ and $g$ of several variables,
$f$ is said to be a minor of $g$  if $f$ can be obtained from $g$ by identifying variables, permuting variables, or adding inessential variables (see Subsection 
\ref{susec:MinorsConstraints}).
 Furthermore, a class is characterizable by a finite number of functional equations if and only if it is characterizable 
 by a finite number of relational constraints (see, e.g., \cite{CF,Pippenger}).
 For the sake of simplifying the presentation of constructions and proofs, we will focus on the approach by relational constraints.

Several properties of functions can be charaterized by relational constraints (or, equivalently, by functional equations.
In fact, uncountably many properties are expressible by such objects, even in the simplest interesting case of 
functions of several variables, i.e., the Boolean functions (see~\cite{CP,Pippenger}).
Classical examples of such properties include idempotency, monotonicity and linearity.
More contemporary examples include submodularity, supermodularity and the combination of the two, i.e., modularity
(see, e.g., \cite{CouMar,Lovasz,Singer,Topkis}).

Another noteworthy example is thresholdness that is the property of those Boolean functions whose 
true points can be separated from the false points by a hyperplane when considered as elements of the $n$-dimensional real space $\IR^n$.
Threshold functions have been widely studied in the literature on Boolean functions, switching theory, 
system reliability theory, game theory, etc.; for background see, e.g., \cite{Isbell,Muroga,Peleg1,Peleg2,Taylor,Winder}.

Despite being a property expressible by relational constraints, thresholdness cannot be captured by a finite set of relational constraints 
(see Hellerstein~\cite{Hellerstein}). 
However, by imposing additional conditions such as linearity or preservation of componentwise conjunctions or disjunctions of tuples, 
the resulting classes of threshold functions may become characterizable by a finite number of relational constraints. 
In fact, these examples  can be obtained from the class of threshold functions by intersecting it with certain clones, 
namely, those of linear functions, conjunctions and disjunctions, respectively.
(Recall that a clone is a class of functions that contains all projections and is closed under functional composition.)
Another noteworthy and well-known example of such an intersection is
the class of ``majority games'', which results as the intersection with the clone of self-dual monotone functions.
The natural question is then: Is the class of majority games characterizable by a finite number of relational constraints?

In this paper we answer negatively to this question. In fact, we will determine, for each clone of Boolean functions,
whether its intersection with the class of threshold functions is finitely characterizable by relational constraints. 
Moreover, we provide finite or infinite characterizing sets of relational constraints accordingly.

The paper is organized as follows.
In the remainder of this section, we recall basic notions and results that will be needed throughout the paper.
The main results are presented in Section~\ref{sec:main}, in particular, the classification of all intersections $C \cap \thr$, where $C$ is a clone and $\thr$ is the class of all threshold functions, as well as the corresponding characterizing set of relational constraints.
For the reader's convenience, the constructions needed for the main results will be left for Section~\ref{sec:constructions}.
One of the main tools in our proof is Taylor and Zwicker's~\cite{Taylor} theorem on the existence of a $k$-asummable function that is not $(k+1)$-asummable.
In Section~\ref{sec:magic}, we slightly refine Taylor and Zwicker's result and show how the classes of functions characterizable by the relational constraints that arise in our current work are related to each other.
Appendix~\ref{App:Post} provides a list of the clones of Boolean functions and relations characterizing them.

\subsection{Boolean functions}

Throughout the paper, we denote the set $\{1, \dots, n\}$ by $\nset{n}$ and the set $\{0, 1\}$ by $\IB$.
A \emph{Boolean function} is a map $f \colon \IB^n \to \IB$ for some positive integer $n$ called the \emph{arity} of $f$.
Typical examples of Boolean functions include
\begin{itemize}
\item the $n$-ary $i$-th \emph{projection} ($i \in \nset{n}$) $e^{(n)}_i \colon \IB^n \to \IB$, $(a_1, \dots, a_n) \mapsto a_i$;
\item \emph{negation} $\comp{\cdot} \colon \IB \to \IB$, $\comp{0} = 1$, $\comp{1} = 0$;
\item \emph{conjunction} $\wedge \colon \IB^2 \to \IB$, $x \wedge y = 1$ if and only if $x = y = 1$;
\item \emph{disjunction} $\vee \colon \IB^2 \to \IB$, $x \vee y = 0$ if and only if $x = y = 0$;
\item \emph{modulo-$2$ addition} $\oplus \colon \IB^2 \to \IB$, $x \oplus y = (x + y) \bmod 2$.
\end{itemize}
The set of all Boolean functions is denoted by $\Omega$ and the set of all projections is denoted by $I_c$.

The preimage $\True{f}$ of $1$ under $f$ is referred to as the set of \emph{true points}, while $\False{f}$ is referred to as the set of \emph{false points}.

The $i$-th variable of a Boolean function $f \colon \IB^n \to \IB$ is said to be \emph{essential} in $f$, or that $f$ depends on $x_i$, 
if there are $a_1, \ldots, a_{i-1},a_{i+1}, \ldots,a_n \in \IB$ such that
\[
f(a_1, \ldots, a_{i-1},0,a_{i+1}, \ldots, a_n) \neq f(a_1, \ldots, a_{i-1}, 1, a_{i+1}, \ldots, a_n).
\]

The \emph{dual} of a Boolean function $f \colon \IB^n \to \IB$ is the function $\dual{f} \colon \IB^n \to \IB$ given by
\[
\dual{f}(x_1, \dots, x_n)
= \comp{f(\comp{x}_1, \dots, \comp{x}_n)}.
\]
A function $f$ is \emph{self-dual} if $f = \dual{f}$.

If $f \colon \IB^n \to \IB$ and $g_1, \dots, g_n \colon \IB^m \to \IB$, then the \emph{composition} of $f$ with $g_1, \dots, g_n$ is the function $f(g_1, \dots, g_n) \colon \IB^m \to \IB$ given by
\[
f(g_1, \dots, g_n)(\aa) = f(g_1(\aa), \dots, g_n(\aa))
\]
for all $\aa \in \IB^m$.
A \emph{clone} of Boolean functions is a subset $C$ of the set $\Omega$ of all Boolean functions that satisfies the following two conditions:
\begin{itemize}
\item $I_c \subseteq C$, i.e., $C$ contains all projections,
\item if $f \colon \IB^n \to \IB$, $g_1, \dots, g_n \colon \IB^m \to \IB$ and $f, g_1, \dots, g_n \in C$, then $f(g_1, \dots, g_n) \in C$, i.e., $C$ is closed under composition.
\end{itemize}

The clones of Boolean functions were completely described by Post~\cite{Post}, and they are often referred to as \emph{Post's classes.} We provide a list of all clones of Boolean functions in Appendix~\ref{App:Post}.

\subsection{Minors and relational constraints}
\label{susec:MinorsConstraints}
We will denote tuples in $\IB^m$ by boldface letters and their entries with corresponding italic letters, e.g., $\mathbf{a} = (a_1, \dots, a_m)$.
Tuples $\mathbf{a} \in \IB^m$ may be viewed as mappings $\mathbf{a} \colon \nset{m} \to \IB$, $i \mapsto a_i$. 
With this convention, given a map $\sigma \colon \nset{n} \to \nset{m}$, we can write the tuple $(a_{\sigma(1)}, \dots, a_{\sigma(n)})$ as $\mathbf{a} \circ \sigma$, or simply $\mathbf{a} \sigma$.

A function $f \colon \IB^m \to \IB$ is a \emph{minor} of another function $g \colon \IB^n \to \IB$ if there exists a map $\sigma \colon \nset{n} \to \nset{m}$ such that $f(\mathbf{a}) = g(\mathbf{a} \sigma)$ for all $\mathbf{a} \in \IB^m$; in this case we write $f \leq g$.
Functions $f$ and $g$ are \emph{equivalent,} denoted $f \equiv g$, if $f \leq g$ and $g \leq f$.
In other words, $f$ is a minor of $g$ if $f$ can be obtained from $g$ by permutation of arguments, addition and deletion of inessential arguments and identification of arguments.
Functions $f$ and $g$ are equivalent if each one can be obtained from the other by permutation of arguments and addition and deletion of inessential arguments.


The minor relation $\leq$ is a quasi-order (i.e., a reflexive and transitive relation) on the set of all Boolean functions,
and the relation $\equiv$ is indeed an equivalence relation. For further background see, e.g., \cite{CL1,CL2,CP,EFHH,Pippenger}

In what follows, we shall consider minors of the following special form.
Let $n \geq 2$, and let $f \colon \IB^n \to \IB$.
For any two-element subset $I$ of $\nset{n}$, we define the function $f_I \colon \IB^{n-1} \to \IB$ by the rule $f_I(\aa) = f(\aa \delta_I)$ for all $\aa \in \IB^{n-1}$, where $\delta_I \colon \nset{n} \to \nset{n-1}$ is given by the rule
\begin{equation}
\label{eq:deltaI}
\delta_I(i) =
\begin{cases}
i, & \text{if $i < \max I$,} \\
\min I, & \text{if $i = \max I$,} \\
i - 1, & \text{if $i > \max I$.}
\end{cases}
\end{equation}
In other words, if $I = \{i, j\}$ with $i < j$, then
\[
f_I(a_1, \dots, a_{n-1}) = f(a_1, \dots, a_{j-1}, a_i, a_j, \dots, a_{n-1}).
\]
Note that $a_i$ occurs twice on the right side of the above equality: both at the $i$-th and at the $j$-th position.
The function $f_I$ will be referred to as an \emph{identification minor} of $f$.


It was shown by Pippenger~\cite{Pippenger} that the classes of functions closed under taking minors are characterizable by so-called relational constraints.
We will briefly survey some results which we will use hereinafter.
An $m$-ary \emph{relational constraint} is a couple $(R, S)$ of $m$-ary relations 
$R$ (the \emph{antecedent}) and $S$ (the \emph{consequent}) on $\IB$ 
(i.e., $R, S \subseteq \IB^m$).  We denote the antecedent and the consequent of a relational constraint 
$Q$ by $R(Q)$ and $S(Q)$, respectively. If both $R(Q)$ and $S(Q)$ equal the binary equality relation, then 
$Q$ is called the binary \emph{equality constraint}. Furthermore, we refer to constraints with empty antecedent and empty consequent 
as \emph{empty constraints}, and to constraints where the antecedent and  consequent are the full relation $\IB^m$, for some $m\geq 1$,
as \emph{full constraints}. The set of all relational constraints is denoted by $\Theta$.

A function $f \colon \IB^n \to \IB$ \emph{preserves} an $m$-ary relational constraint $(R, S)$, denoted $f \triangleright (R, S)$, if for every $\mathbf{a}^{1}, \dots, \mathbf{a}^{n} \in R$, we have $f(\mathbf{a}^1, \dots, \mathbf{a}^n) \in S$. (Regarding tuples $\mathbf{a}^i$ as unary maps, $f(\mathbf{a}^1, \dots, \mathbf{a}^n)$ denotes the $m$-tuple whose $i$-th entry is $f(\mathbf{a}^1, \dots, \mathbf{a}^n)(i) = f(a^1_i, \dots, a^n_i)$.)

The preservation relation gives rise to a Galois connection between functions and relational constraints that we now briefly describe; for further background, 
see~\cite{C,CP,Pippenger}.
Define $\Pol \colon \mathcal{P}(\Theta) \to \mathcal{P}(\Omega)$, $\Inv \colon \mathcal{P}(\Omega) \to \mathcal{P}(\Theta)$ by
\begin{align*}
\Pol(\mathcal{Q}) &= \{f \in \Omega : \text{$f \triangleright Q$ for every $Q \in \mathcal{Q}$}\}, \\
\Inv(\mathcal{F}) &= \{Q \in \Theta : \text{$f \triangleright Q$ for every $f \in \mathcal{F}$}\}.
\end{align*}
We say that a set $\mathcal{F}$ of functions is \emph{characterized} by a set $\mathcal{Q}$ of relational constraints if 
$\mathcal{F} = \Pol(\mathcal{Q})$.
Dually, $\mathcal{Q}$ is \emph{characterized} by $\mathcal{F}$ if $\mathcal{Q} = \Inv(\mathcal{F})$.
In other words, sets of functions characterizable by relational constraints are exactly the fixed points of 
$\Pol \circ \Inv$, and, dually, sets of relational constraints characterizable by functions are exactly the fixed points of $\Inv \circ \Pol$.

\begin{remark}
Preservation of a relational constraint generalizes the notion of preservation of a relation, as in the classical $\mathrm{Pol}$--$\mathrm{Inv}$ theory of clones and relations, which establishes that the clones on finite sets are exactly the classes of functions that are characterized by relations (see~\cite{BKKR,Geiger}).
In this framework, a function $f$ preserves a relation $R$ if and only if $f$ preserves the relational constraint $(R, R)$.
Hence, clones are exactly the classes that are characterized by relational constraints of the form $(R, R)$ for some relation $R$.
\end{remark}

The following result reassembles various descriptions of the Galois closed sets of functions, which can be found in \cite{CP,EFHH,Pippenger}.

\begin{theorem}
\label{thm:characterizablefunctions}
Let $\mathcal{F}$ be a set of functions. The following are equivalent.
\begin{enumerate}[label=\rm (\roman*)]
\item\label{thm:characterizablefunctions:item1}
$\mathcal{F}$ is closed under taking minors.
\item\label{thm:characterizablefunctions:item2}
$\mathcal{F}$ is characterizable by relational constraints.
\item\label{thm:characterizablefunctions:item3}
$\mathcal{F}$ is of the form
\[
\forbid(A) := \{f \in \Omega : \text{$g \nleq f$ for all $g \in A$}\}
\]
for some antichain $A$ with respect to the minor relation $\leq$.
\end{enumerate}
\end{theorem}

\begin{remark}
It follows from the equivalence of \ref{thm:characterizablefunctions:item1} and \ref{thm:characterizablefunctions:item2} in Theorem~\ref{thm:characterizablefunctions} that the union and the intersection of classes that are characterizable by relational constraints are characterizable by relational constraints.
\end{remark}

\begin{remark}
Note that the antichain $A$ in item \ref{thm:characterizablefunctions:item3} of Theorem~\ref{thm:characterizablefunctions} is unique up to equivalence. In fact, 
$A$ can be chosen among the minimal elements of $\Omega \setminus \mathcal{F}$; the elements of $A$ are 
called \emph{minimal forbidden minors for $\mathcal{F}$}.
\end{remark}

As we will see, there are classes of functions that, even though characterizable by relational constraints, are not characterized by any finite set of relational constraints. 
A set of functions is \emph{finitely characterizable} if it is characterized by a finite set of relational constraints.

The following theorem is a refinement of Theorem~\ref{thm:characterizablefunctions} and provides a description for finitely characterizable classes.

\begin{theorem}[\cite{CP,EFHH}]
\label{thm:finitelycharacterizablefunctions}
Let $\mathcal{F}$ be a set of functions. The following are equivalent.
\begin{enumerate}[label=\rm (\roman*)]
\item $\mathcal{F}$ is finitely characterizable.
\item $\mathcal{F}$ is of the form $\forbid(A)$ for some finite antichain $A$ with respect to the minor relation $\leq$.
\end{enumerate}
\end{theorem}

The Galois closed sets of relational constraints were likewise described by Pippenger~\cite{Pippenger}; this description was 
extended to arbitrary, possibly infinite, underlying sets in \cite{CF}. We shall briefly survey Pippenger's description of the 
Galois closed sets of constraints. 

An $m$-ary relational constraint $(R,S)$ is a \emph{simple minor} of an $(m+p)$-ary relational constraint $(R',S')$ if there is 
$h \colon \{1, \ldots, n\} \to \{1, \ldots, m+p\}$ such that 
\begin{align*} 
  R
   \left(\begin{array}{c}
   x_1\\
   \vdots \\
   x_m\\
   \end{array}\right)
  &
\quad  \iff  \quad 
\exists x_{m+1} \ldots \exists x_{m+p} \quad 
R'
   \left(\begin{array}{c}
   x_{h(1)}\\
   \vdots \\
   x_{h(n)}\\
   \end{array}\right)
\intertext{and}
  S
   \left(\begin{array}{c}
   x_1\\
   \vdots \\
   x_m\\
   \end{array}\right)
  &
\quad \iff \quad 
\exists x_{m+1} \ldots \exists x_{m+p} \quad 
S'
   \left(\begin{array}{c}
   x_{h(1)}\\
   \vdots \\
   x_{h(n)}\\
   \end{array}\right).
 \end{align*}

Note that simple minors subsume the notions of permutation, diagonalization and projection of arguments; for background
see \cite{CF,Pippenger}. 

A constraint $(R,S)$ is obtained from a constraint $(R',S)$ by \emph{restricting the antecedent} if
$R\subseteq R'$. Likewise,  $(R,S)$ is obtained from a constraint $(R,S')$ by \emph{extending the consequent} if
$S\supseteq S'$. A constraint $(R,S\cap S')$ is said to be obtained from $(R,S)$ and $(R, S')$ by intersecting consequents.

A set $\mathcal{Q}$ of relational constraints is said to be \emph{minor-closed} if it contains the binary equality constraint,
the unary empty constraint, and it is closed under taking simple minors, restricting antecedents, and extending  and intersecting consequents.

We can now state Pippenger's \cite{Pippenger} description of the Galois closed sets of relational constraints.

\begin{theorem}
\label{thm:characterizableconstraints}
Let $\mathcal{Q}$ be a set of relational constraints. The following are equivalent.
\begin{enumerate}[label=\rm (\roman*)]
\item\label{thm:characterizableconstraints:item1}
 $\mathcal{Q}$ is characterizable by some set of functions.
\item\label{thm:characterizableconstraints:item2}
 $\mathcal{Q}$ is minor-closed.
 \end{enumerate}
\end{theorem}

The following lemma provides a noteworthy tool for showing that certain classes of
threshold functions are not finitely characterizable. 

\begin{lemma} \label{lemma:finiteDescendingChainIffFinitelyCharacterizable}
Let $C$ and $C_i$ for all $i \geq 1$ be classes of functions that are closed under taking minors, such that 
$C  = \bigcap_{i \geq 1} C_i$,
and
$C_{i+1} \subseteq C_i$ for all $i \geq 1$.
If $C$ is finitely characterizable by constraints, then
there exists $\ell \in \IN$ such that $C_j = C_\ell$ for all $j \geq \ell$.
\end{lemma}

\begin{proof}
  By Theorem~\ref{thm:characterizablefunctions}, each minor-closed class $C_i$ is characterized by some 
  set $\mathcal{Q}_i$ of relational constraints, i.e., $C_i = \Pol \mathcal{Q}_i$ for all $i \geq 1$.

  Assume that $C$ is finitely characterizable. 
  Then there is some finite set $\mathcal{P}$ of constraints with $C = \Pol \mathcal{P}$, and thus
  $\Inv C = \Inv \Pol \mathcal{P}$. 
  Since 
  \[ C = \bigcap_{i \geq 1} \Pol \mathcal{Q}_i = \Pol \bigcup_{i \geq 1} \mathcal{Q}_i \]
  we can construct each $P \in \mathcal{P}$ from
  the constraints in $\bigcup_{i \geq 1} \mathcal{Q}_i$. Since the constraints are finite, 
  all such constructions are finite. In particular,
  only a finite number of constraints from 
  $\bigcup_{i \geq 1} \mathcal{Q}_i$ are used for each $P \in \mathcal{P}$. Since $\mathcal{P}$ is finite,
  this implies that only a finite number of constraints are needed to construct all $P \in \mathcal{P}$.
  Therefore 
\[
C = \Pol \mathcal{P} \supseteq \Pol \bigcup_{i = 1}^l \mathcal{Q}_i
= \bigcap_{i=1}^l \Pol \mathcal{Q}_i
= \Pol \mathcal{Q}_l = C_l
\]
  holds for some $l \in \IN$. Now this implies that for any $j \geq l$,
\[
C \subseteq C_j \subseteq C_l \subseteq C
\]
  and consequently $C = C_j = C_l$ for all $j \geq l$.
\end{proof}


\section{Main results: classification and characterizations of Galois closed sets of threshold functions}
\label{sec:main}

\subsection{Motivation}

A \emph{threshold function} is a Boolean function $f \colon \IB^n \to \IB$ such that there exist \emph{weights} $w_1, \dots, w_n \in \IR$ and a \emph{threshold} $t \in \IR$ fulfilling
\[
f(x_1, \dots, x_n) = 1 \iff \sum_{i=1}^n w_i x_i \geq t.
\]
Another, equivalent, definition is the following.
An $n$-ary Boolean function $f$ is called a \emph{threshold function} if there is a 
hyperplane in $\IR^n$ strictly separating the true points of $f$ from the false points of $f$, considered as elements of $\IR^n$.
The set of all threshold functions is denoted by $\thr$.

The class of threshold functions has remarkable invariance properties.
For instance, it is closed under taking negations and duals (see Lemma~\ref{lem:dualBl}).
%
%
Moreover, the class of threshold functions is also closed under taking minors of its members; 
hence it is characterizable by relational constraints by Theorem~\ref{thm:characterizablefunctions}.
However, no finite set of relational constraints suffices.

\begin{theorem}[Hellerstein~\cite{Hellerstein}]
\label{thm:Hell}
The class of threshold functions is not finitely characterizable. 
\end{theorem}

Imposing some additional conditions on threshold functions, we may obtain proper subclasses of $\thr$ that are finitely characterizable.
Easy examples arise from the intersections of $\thr$ with the clones $L$, $\Lambda$, $V$ (see Appendix~\ref{App:Post}).
However, as we have seen, other intersections $C \cap T$ may fail to be finitely characterizable, e.g., for $C = \Omega$.

This fact gives rise to the following problem.

\begin{problem}
Which clones $C$ of Boolean functions have the property that $C \cap \thr$ is finitely characterizable?
\end{problem}

In the following subsection we present a solution to this problem.

\subsection{Classification and characterizations of intersections of the class of threshold functions with clones}

We start by observing that 
\[
L \cap \thr = \Omega(1), \qquad
\Lambda \subseteq \thr, \qquad
V \subseteq \thr,
\]
from which it follows that the intersection $C \cap \thr$ is a clone for any clone $C$ contained in one of $L$, $V$ and $\Lambda$.
Hence, the characterization of $C \cap \thr$ for any such clone $C$
is given by the relational constraint $(R, R)$, where $R$ is the relation characterizing $C \cap \thr$ given in Appendix~\ref{App:Post}.

We proceed to characterizing the intersections $C \cap \thr$ for the remaining clones $C$; as we will see, none of these is finitely characterizable.
A characterization of the class $\thr$ of all threshold functions (i.e., for $C = \Omega$) is easily obtained with the help of the notion of asummability.

For $k \geq 2$, a Boolean function $f \colon \IB^n \to \IB$ is \emph{$k$-asummable} if for any $m \in \{2, \dots, k\}$ and for all $\aa_1, \dots, \aa_m \in f^{-1}(0)$ and $\bb_1, \dots, \bb_m \in f^{-1}(1)$, it holds that
\[
\aa_1 + \dots + \aa_m \neq \bb_1 + \dots + \bb_m.
\]
(Addition here is standard vector addition in $\IR^n$.)
A function is \emph{asummable} if it is $k$-asummable for all $k \geq 2$.
It is well known that asummability characterizes threshold functions; see~\cite{Chow,Elgot,Muroga}.

\begin{theorem}
\label{thm:thresholdasummable}
A Boolean function is threshold if and only if it is asummable.
\end{theorem}






Define for $n \geq 1$, the $2n$-ary relational constraint $B_n$ as
\begin{align*}
R(B_n) & := \{ (x_1, \dots, x_{2n}) \in \IB^{2n} : \sum_{i=1}^n x_i = \sum_{i=n+1}^{2n} x_i \} \\
S(B_n) & := \IB^{2n} \setminus \{ (\underbrace{0,\dots,0}_{n},\underbrace{1,\dots,1}_n), 
                                  (\underbrace{1,\dots,1}_{n},\underbrace{0,\dots,0}_n) \}. 
\end{align*}
Note that in the definition of $R(B_n)$ we employ the usual addition of real numbers. Denoting by $w(\aa)$ the \emph{Hamming weight} of a tuple $\aa \in \IB^n$ (i.e., the number of nonzero entries in $\aa$), we can equivalently define $R(B_n)$ as $\{(x_1, \dots, x_{2n}) \in \IB^{2n} : w(x_1, \dots, x_n) = w(x_{n+1}, \dots, x_{2n})\}$.

\begin{lemma}
\label{lem:PolBl}
Let $f \colon \IB^n \to \IB$ and $\ell \geq 2$.
Then $\aa_1 + \dots + \aa_\ell \neq \bb_1 + \dots + \bb_\ell$ for all $\aa_1, \dots, \aa_\ell \in f^{-1}(0)$ and $\bb_1, \dots, \bb_\ell \in f^{-1}(1)$ if and only if $f$ preserves $B_\ell$.
\end{lemma}

\begin{proof}
Assume first that $f$ does not preserve $B_\ell$. Then there exists a matrix
\[
M = \begin{pmatrix}
  m_1^1 & m_2^1 & \dots & m_n^1 \\
  m_1^2 & m_2^2 & \dots & m_n^2 \\
  \vdots & \vdots &  & \vdots \\
  m_1^{2 \ell} & m_2^{2 \ell} & \dots & m_n^{2 \ell}
\end{pmatrix} = \begin{pmatrix} 
  M^1 \\
  M^2 \\
  \vdots \\
  M^{2 \ell}
\end{pmatrix}
= (M_1, M_2, \dots, M_n),
\]
i.e., $M^1, \dots, M^{2 \ell} \in \IB^n$ are the rows of $M$, 
and $M_1, \dots, M_n \in \IB^{2 \ell}$ are the columns of $M$, such that
\begin{itemize}
\item $M_1, \dots, M_n \in R(B_\ell)$, and
\item $\zz := g(M_1, \dots, M_n) := \begin{pmatrix} g(M^1) \\ \vdots \\ g(M^{2 \ell}) \end{pmatrix} \notin S(B_\ell)$.
\end{itemize}
Thus $\zz \in \{ (\underbrace{0,\dots,0}_l,\underbrace{1,\dots,1}_l), 
(\underbrace{1,\dots,1}_l,\underbrace{0,\dots,0}_l)\}$.
As $B_\ell$ is invariant under swapping the first $\ell$ rows with the last $\ell$ rows, we can
assume that $\zz = (\underbrace{0, \dots, 0}_\ell, \underbrace{1, \dots, 1}_\ell)$.
Then $M^1, \dots, M^\ell \in f^{-1}(0)$ and $M^{\ell + 1}, \dots, M^{2 \ell} \in f^{-1}(1)$, and
$M^1 + \dots + M^\ell = M^{\ell + 1} + \dots + M^{2 \ell}$ by the definition of $B_\ell$.

Assume then that there exist $\aa_1, \dots, \aa_\ell \in f^{-1}(0)$ and $\bb_1, \dots, \bb_\ell \in f^{-1}(1)$ such that $\aa_1 + \dots + \aa_\ell = \bb_1 + \dots + \bb_\ell$.
Let $M$ be the $2 \ell \times n$ matrix whose rows are $\aa_1, \dots, \aa_\ell, \bb_1, \dots, \bb_\ell$. The columns of $M$ are tuples in $R(B_\ell)$, but $f(M) = (\underbrace{0, \dots, 0}_\ell, \underbrace{1, \dots, 1}_\ell) \notin S(B_\ell)$. We conclude that $f$ does not preserve $B_\ell$.
\end{proof}

Now it is easy to define a set of relational constraints that characterizes $k$-asummable functions.
For $k \geq 2$, let $\mathcal{A}_k := \{B_n : 2 \leq n \leq k\}$.

\begin{lemma}
\label{lem:k-asummable}
Let $k \geq 2$. A Boolean function $f$ is $k$-asummable if and only if $f \in \Pol(\mathcal{A}_k)$.
\end{lemma}

\begin{proof}
Follows immediately from the definition of $k$-asummability and Lemma~\ref{lem:PolBl}.
\end{proof}

\begin{corollary}
Let $f \colon \IB^n \to \IB$. The following are equivalent.
\begin{enumerate}[label=\rm (\roman*)]
\item\label{cor:thr:item:thr}
$f$ is a threshold function.
\item\label{cor:thr:item:Ak}
$f \in \bigcap_{k \geq 2} \Pol(\mathcal{A}_k)$.
\item\label{cor:thr:item:Bn}
$f \in \Pol (\{B_n : n \geq 2\})$.
\end{enumerate}
\end{corollary}

\begin{proof}
The equivalence of \ref{cor:thr:item:thr} and \ref{cor:thr:item:Ak} follows immediately from Theorem~\ref{thm:thresholdasummable} and Lemma~\ref{lem:k-asummable}. Conditions \ref{cor:thr:item:Ak} and \ref{cor:thr:item:Bn} are equivalent, because
\[
\bigcap_{k \geq 2} \Pol(\mathcal{A}_k)
= \Pol \Bigl( \bigcup_{k \geq 2} \mathcal{A}_k \Bigr)
= \Pol \Bigl( \bigcup_{k \geq 2} \{B_n : 2 \leq n \leq k\} \Bigr)
= \Pol (\{B_n : n \geq 2\}).
\qedhere
\]
\end{proof}

Since $\mathcal{A}_k \subseteq \mathcal{A}_k \cup \{B_{k+1}\} = \mathcal{A}_{k+1}$, it is clear that $\Pol(\mathcal{A}_{k+1}) \subseteq \Pol(\mathcal{A}_k)$ for all $k \geq 2$.
Taylor and Zwicker have shown in \cite{Taylor} that for every $k \geq 2$, there exist $k$-asummable functions that are not $(k+1)$-asummable. 
Hence these inclusions are strict for every $k$.

\begin{theorem}
\label{theorem:strictAkInclusions}
For all $k \geq 2$, $\Pol(\mathcal{A}_{k+1}) \subset \Pol(\mathcal{A}_k)$.
\end{theorem}

\begin{theorem}
\label{thm:characterization}
The set $\Pol (\{ B_n : n \geq 2 \})$ is the class of all threshold functions. 
Moreover, for every clone $C$, the subclass $C \cap \thr$ of threshold functions is characterized by the set 
$\{B_n : n \geq 2\} \cup \mathcal{Q}_C$, where $\mathcal{Q}_C$ is the set of relational constraints characterizing the clone $C$, as given in Appendix~\ref{App:Post}.
\end{theorem}


\begin{remark}
From Theorems  \ref{theorem:strictAkInclusions} and \ref{thm:characterization} it follows that
\[
\thr = \bigcap_{k \geq 2} \Pol(\mathcal{A}_k) \subset \dots \subset \Pol(\mathcal{A}_{\ell+1}) \subset \Pol(\mathcal{A}_\ell) \subset \dots \subset \Pol(\mathcal{A}_2)
\]
holds for all $\ell \geq 3$, i.e., the sets $\Pol(\mathcal{A}_k)$ with $k \geq 2$ form an infinite descending chain, whose
intersection is the set $\thr$ of all threshold functions.
 \end{remark}

Theorem~\ref{thm:characterization} provides an infinite set of relational constraints characterizing the set $C \cap \thr$ for each clone $C$. 
As Theorem~\ref{thm:classification} will reveal, the characterization provided is optimal for 
the clones not contained in $L$, $V$ or $\Lambda$ in the sense that for such clones $C$, 
the set $C \cap \thr$ is not finitely characterizable by relational constraints.

In order to proceed, we need the following lemma. Its proof is somewhat technical and is deferred to Section~\ref{sec:constructions}.

\begin{lemma}
\label{lem:GCf}

Let $f$ be a Boolean function, and let $C \in \{SM, M_cU_\infty, M_cW_\infty\}$. There exists a Boolean function $G_C(f)$ that satisfies the following conditions:
\begin{enumerate}[label=\rm (\roman*)]
\item $G_C(f) \in C$,
\item for all $n \geq 2$, $f \in \Pol B_n$ if and only if $G_C(f) \in \Pol B_n$.
\end{enumerate}
\end{lemma}

\begin{proof}
This brings together Corollaries~\ref{cor:GSM}, \ref{cor:GMcUinfty} and \ref{cor:GMcWinfty}, which will be proved in Section~\ref{sec:constructions}.
\end{proof}






\begin{remark}\label{rem:OldTheorem-theorem:strictCcapAkInclusions}
 Lemma \ref{lem:GCf} gives rise to a noteworthy refinement of Theorem~\ref{theorem:strictAkInclusions}. 
 Indeed, 
by Theorem~\ref{theorem:strictAkInclusions}, there is some $f \in \Pol(\mathcal{A}_k) \setminus \Pol(\mathcal{A}_{k+1})$ and, 
by Lemma~\ref{lem:GCf}, there exists a function $G_E(f) \in E \subseteq C$ satisfying 
$G_{E}(f) \in \Pol(\mathcal{A}_k) \setminus \Pol(\mathcal{A}_{k+1})$.
This implies that $G_{E}(f) \in (C \cap \Pol(\mathcal{A}_k)) \setminus (C \cap \Pol(\mathcal{A}_{k+1}))$ and thus
 $C \cap \Pol(\mathcal{A}_{k+1}) \subset C \cap \Pol(\mathcal{A}_k)$ for all $k \geq 2$.
 
 This shows that if $C$ is a clone of Boolean functions satisfying $E \subseteq C$ for some $E \in \{ SM, M_cU_\infty, M_cW_\infty \}$,
then $C \cap \Pol(\mathcal{A}_{k+1}) \subset C \cap \Pol(\mathcal{A}_k)$ for all $k \geq 2$.
\end{remark}

%

\begin{figure}[t]
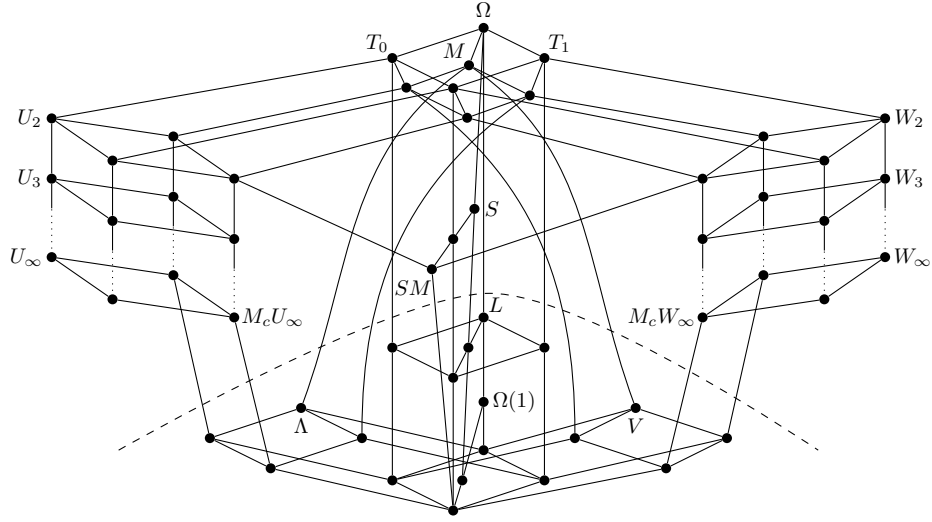

  \PostsLattice{0.8}
  \caption{Post's lattice. Illustration of Theorem~\ref{thm:classification}:
    for a clone $C$, the set $C \cap \thr$ of threshold functions in $C$ is finitely characterizable 
    if and only if $C$ is below the dashed line.}
    \label{figure:PostsLattice}
\end{figure}

\begin{theorem}
\label{thm:classification}
Let $C$ be a clone of Boolean functions.
The subclass $C \cap \thr$ of threshold functions is finitely characterizable if and only if $C$ is contained in one of the clones $L$, $V$, $\Lambda$.
\end{theorem}

This theorem is illustrated by Figure~\ref{figure:PostsLattice}.

\begin{proof}
We have already observed that $C \cap \thr$ is finitely characterizable for every subclone $C$ of $L$, $V$ or $\Lambda$.

Now we consider all the other clones. Let $C$ be a clone such that $C \not\subseteq D$ for all $D \in \{L, V, \Lambda\}$.
We can read off of Post's lattice (see Figure~\ref{figure:PostsLattice}) that there is some $E \in \{SM, M_cU_\infty, M_cW_\infty\}$ such that $E \subseteq C$.
It follows from Theorem~\ref{theorem:strictAkInclusions} and Lemma~\ref{lem:GCf} that for every $k \geq 2$, there exists a function $f_k \in E$ such that $f_k \in \Pol B_\ell$ whenever $2 \leq \ell \leq k$ and $f_k \notin \Pol B_{k+1}$.
Note that $f_k \notin C \cap \thr$.
 
Suppose, on the contrary that $C \cap \thr$ is finitely characterizable.
By Theorem~\ref{thm:finitelycharacterizablefunctions}, $C \cap \thr$ is of the form $\forbid(A)$ for some finite antichain $A$ of minimal forbidden minors.
Each one of the functions $f_k$ has a minor in $A$.
Since $A$ is finite, there is an element $g \in A$ and an infinite set $S \subseteq \IN$ such that $g \leq f_k$ for all $k \in S$.
The function $g$ is not threshold, so there exists $p \in \IN$ such that $p \geq 2$ and $g \notin \Pol B_p$.
Being infinite, the set $S$ contains an element $q$ with $p \leq q$.
Then we have $g \leq f_q$ and $f_q \in \Pol B_p$.
We also have $g \in \Pol B_p$, because $\Pol B_p$ is closed under taking minors.
This yields the desired contradiction.
\end{proof}

\begin{remark}
 Alternatively, Theorem~\ref{thm:classification} can be proved using Lemma~\ref{lemma:finiteDescendingChainIffFinitelyCharacterizable}
 and Remark \ref{rem:OldTheorem-theorem:strictCcapAkInclusions}.

  As before if $C$ is a subclone of $L$, $V$ or $\Lambda$, then $C \cap \thr$ is finitely characterizable.
  As for any other clone $C$, we know (once again reading off of Post's lattice) that there is some 
  $E \in \{SM, M_cU_\infty, M_cW_\infty\}$ such that $E \subseteq C$.
By Remark \ref{rem:OldTheorem-theorem:strictCcapAkInclusions}, we have 
 $C \cap \Pol(\mathcal{A}_{k+1}) \subset C \cap \Pol(\mathcal{A}_k)$ for all $k \geq 2$.
 Furthermore,
 \[
 C \cap \thr
 = C \cap \bigcap_{n \geq 2} \Pol(\mathcal{A}_k)
 = \bigcap_{n \geq 2} (C \cap \Pol(\mathcal{A}_k)),
 \]
 i.e., we have an infinite descending chain the intersection of which equals $C \cap \thr$.
 By Lemma~\ref{lemma:finiteDescendingChainIffFinitelyCharacterizable} we thus conclude that $C \cap \thr$
 is not finitely characterizable.
 \end{remark}


\section{Constructions}
\label{sec:constructions}

In order to prove Lemma~\ref{lem:GCf}, we will construct from a given Boolean function $f$, for each $C \in \{S, M_c, SM, U_\infty, M_cU_\infty, M_cW_\infty \}$, a Boolean function $G_C(f)$ that satisfies the following conditions:
\begin{enumerate}[label=\rm (\roman*)]
\item $G_C(f) \in C$,
\item for all $\ell \geq 2$, $f \in \Pol B_\ell$ if and only if $G_C(f) \in \Pol B_\ell$.
\end{enumerate}
We do this step by step. We first construct functions $G_S(f)$ and $G_{M_c}(f)$ with the desired properties.
Using these two constructions as building blocks, we can construct $G_{SM}$ as $G_{M_c}(G_S(f))$.
Then we construct $G_{U_\infty}(f)$, and, building upon this, we finally get $G_{M_cU_\infty}(f) := G_{U_\infty}(G_{M_c}(f))$ and $G_{M_cW_\infty}(f) := \dual{(G_{M_cU_\infty}(f))}$.




\subsection{Construction of $G_{S}(f)$}
\label{subsec:ST}


Let $f \colon \IB^n \to \IB$. Then we define $G_{S}(f) \colon \IB^{n+1} \to \IB$ by
\[
G_{S}(f)(x_1, \dots, x_{n+1}) = (x_{n+1} \wedge f(x_1, \dots, x_n)) \vee (\comp{x}_{n+1} \wedge \dual{f}(x_1, \dots, x_n)).
\]

\begin{lemma} \label{lemma:anytoS}
For any $f \colon \IB^n \to \IB$, the function $G_{S}(f)$ is self-dual.
\end{lemma}

\begin{proof}
  Let $g := G_{S}(f)$.
  Then
  \begin{align*}
  \dual{g}(\mathbf{x}, x_{n+1}) &=
  \comp{(\comp{x}_{n+1} \wedge f(\comp{\mathbf{x}})) \vee (\comp{\comp{x}}_{n+1} \wedge \dual{f}(\comp{\mathbf{x}}))} \\
  &= (x_{n+1} \vee \dual{f}(\mathbf{x})) \wedge (\comp{x}_{n+1} \vee f(\mathbf{x})) \\
  &= (x_{n+1} \wedge \comp{x}_{n+1}) \vee (x_{n+1} \wedge f(\mathbf{x})) \vee (\dual{f}(\mathbf{x}) \wedge \comp{x}_{n+1}) \vee (\dual{f}(\mathbf{x}) \wedge f(\mathbf{x})) \\
  &= (x_{n+1} \wedge f(\mathbf{x})) \vee (\dual{f}(\mathbf{x}) \wedge \comp{x}_{n+1}) \\
  &= g(\mathbf{x}, x_{n+1}),
  \end{align*}
  where the second last equality holds since
\[
\dual{f}(\mathbf{x}) \wedge f(\mathbf{x}) \leq (x_{n+1} \wedge f(\mathbf{x})) \vee (\dual{f}(\mathbf{x}) \wedge \comp{x}_{n+1})
\]
for every $x_{n+1}$.
\end{proof}

\begin{lemma} \label{lemma:ST:fNotinBlimpliesgNotinBl}
Let $f \colon \IB^n \to \IB$.
If $f \notin \Pol B_\ell$ for some $\ell \geq 2$, then $G_{S}(f) \notin \Pol B_\ell$.
\end{lemma}

\begin{proof}
Assume that $f \notin \Pol B_\ell$, and let $g := G_{S}(f)$.
Then there are $\yy_1, \dots, \yy_n \in R(B_\ell)$ with $f(\yy_1, \dots, \yy_n) \notin S(B_\ell)$.
Since $g(x_1, \dots, x_n, 1) = f(x_1, \dots, x_n)$, we have
\[
g(\yy_1,\dots,\yy_n,\11) = f(\yy_1,\dots,\yy_n) \notin S(B_\ell).
\]
Since also $\11 \in R(B_\ell)$, we conclude that $g \notin \Pol B_\ell$.
\end{proof}

\begin{lemma} \label{lemma:ST:fInBlImpliesgInBl}
Let $f \colon \IB^n \to \IB$.
If $f \in \Pol B_\ell$ for some $\ell \geq 2$, then $G_{S}(f) \in \Pol B_\ell$.
\end{lemma}

\begin{proof}
Let $g := G_{S}(f)$.
Suppose, on the contrary, that $g \notin \Pol B_\ell$. Then there is some matrix $M$ given by
  \[
    M = \begin{pmatrix}
      m_1^1 & m_2^1 & \dots & m_{n+1}^1 \\
      m_1^2 & m_2^2 & \dots & m_{n+1}^2 \\
      \vdots & \vdots &  & \vdots \\
      m_1^{2\ell} & m_2^{2\ell} & \dots & m_{n+1}^{2\ell}
    \end{pmatrix} = \begin{pmatrix} 
      M^1 \\
      M^2 \\
      \vdots \\
      M^{2\ell}
    \end{pmatrix}
    = (M_1,M_2,\dots,M_{n+1}),
  \]
  i.e., $M^1,\dots,M^{2\ell} \in \IB^{n+1}$ are the rows of $M$, 
  and $M_1,\dots,M_{n+1} \in \IB^{2\ell}$ are the columns of $M$, such that
  \begin{itemize}
    \item $M_1,\dots,M_{n+1} \in R(B_\ell)$, and
    \item $\zz := g(M_1,\dots,M_{n+1}) := \begin{pmatrix} g(M^1) \\ \vdots \\ g(M^{2\ell}) \end{pmatrix} 
        \notin S(B_\ell)$.
  \end{itemize}
  Thus $\zz \in \{ (\underbrace{0,\dots,0}_\ell,\underbrace{1,\dots,1}_\ell), 
  (\underbrace{1,\dots,1}_\ell, \underbrace{0,\dots,0}_\ell)\}$.
  As $B_\ell$ is invariant under swapping the first $\ell$ coordinates with the last $\ell$ coordinates, we can
  assume that $\zz = (\underbrace{0,\dots,0}_\ell, \underbrace{1,\dots,1}_\ell)$.

  We now look at the last column $M_{n+1}$ of $M$. Since $\sum_{i=1}^\ell m_{n+1}^i = \sum_{i=\ell+1}^{2\ell} m_{n+1}^i$,
  and since $B_\ell$ is totally symmetric on the first $\ell$ rows and on the last $\ell$ rows, we can assume
  that \[
    M_{n+1} = (\underbrace{0,\dots,0}_\alpha, \underbrace{1,\dots,1}_\beta,
               \underbrace{0,\dots,0}_\alpha, \underbrace{1,\dots,1}_\beta)
  \]
  holds for some $\alpha,\beta \geq 0$ with $\alpha+\beta = \ell$.

  We will now construct a matrix $K$ with
  \[
    K = \begin{pmatrix}
      k_1^1 & k_2^1 & \dots & k_n^1 \\
      k_1^2 & k_2^2 & \dots & k_n^2 \\
      \vdots & \vdots &  & \vdots \\
      k_1^{2\ell} & k_2^{2\ell} & \dots & k_n^{2\ell}
    \end{pmatrix} = \begin{pmatrix} 
      K^1 \\
      K^2 \\
      \vdots \\
      K^{2\ell}
    \end{pmatrix}
    = (K_1,K_2,\dots,K_n),
  \]
that satisfies $K_1, \dots, K_n \in R(B_\ell)$ and $f(K_1, \dots, K_n) \notin S(B_\ell)$. 
This will yield the desired contradiction since we started with the assumption that $f \in \Pol B_\ell$.

  We define $k_j^i$ for $1 \leq i \leq 2\ell$ and $1 \leq j \leq n$ by
  \[ \def\arraystretch{1.2}
    k_j^i = \left\{\begin{array}{llr@{{} \leq i \leq {}}l}
      \overline{m}_j^{i+\ell} & \text{if } & 1          & \alpha \\
      m_j^i                & \text{if } & \alpha+1   & \ell \\
      \overline{m}_j^{i-\ell} & \text{if } & \ell+1        & \ell+\alpha \\
      m_j^i                & \text{if } & \ell+\alpha+1 & 2\ell
    \end{array}\right.
  \]
In other words, matrix $K$ is obtained from $M$ by omitting the last column, negating rows $1, \dots, \alpha$ and $\ell + 1, \dots, \ell + \alpha$, and then swapping rows $1, \dots, \alpha$ with rows $\ell + 1, \dots, \ell + \alpha$.
  
  We need to show that $K_j \in R(B_\ell)$ for all $j \in [n]$.
  Let $j \in [n]$ be arbitrary, and let 
  \[ 
    a := \sum_{i=1}^\alpha m_j^i, \qquad
    b := \sum_{i=\alpha+1}^\ell m_j^i, \qquad
    c := \sum_{i=\ell+1}^{\ell+\alpha} m_j^i, \qquad
    d := \sum_{i=\ell+\alpha+1}^{2\ell} m_j^i. 
  \]
  Since $M_j \in R(B_\ell)$ we have 
  \begin{equation} \label{equation:ST:abEqualscd}
    a+b = \sum_{i=1}^\ell m_j^i = \sum_{i=\ell+1}^{2\ell} m_j^i = c + d.
  \end{equation}
  For $K_j$ we find the following:
  \begin{align*}
    \sum_{i=1}^\alpha k_j^i & 
     = \sum_{i=1}^\alpha \overline{m}_j^{i+\ell} 
     = \sum_{i=1}^\alpha (1-m_j^{i+\ell}) 
     = \alpha - \sum_{i=\ell+1}^{\ell+\alpha} m_j^i 
     = \alpha - c, \displaybreak[0]\\
    \sum_{i=\alpha+1}^\ell k_j^i & = \sum_{i=\alpha+1}^\ell m_j^i = b, \displaybreak[0]\\
    \sum_{i=\ell+1}^{\ell+\alpha} k_j^i & 
     = \sum_{i=\ell+1}^{\ell+\alpha} \overline{m}_j^{i-\ell}
     = \sum_{i=\ell+1}^{\ell+\alpha} (1-m_j^{i-\ell})
     = \alpha - \sum_{i=1}^{\alpha} m_j^i 
     = \alpha - a, \displaybreak[0]\\
    \sum_{i=\ell+\alpha+1}^{2\ell} k_j^i & = \sum_{i=\ell+\alpha+1}^{2\ell} m_j^i = d.
  \end{align*}
  From this it follows that
  \[ \sum_{i=1}^\ell k_j^i  = \alpha - c + b \overset{\eqref{equation:ST:abEqualscd}}{=} 
    \alpha - a + d = \sum_{i=\ell+1}^{2\ell} k_j^i, \]
  and thus $K_j \in R(B_\ell)$ for all $j \in [n]$.

  We now show that $f(K^i) = 0$ if $1 \leq i \leq \ell$, and $f(K^i) = 1$ if $\ell+1 \leq i \leq 2\ell$.
We need to consider four different cases for $i$:
  \begin{itemize}
    \item $1 \leq i \leq \alpha$. Then $(\overline{K}^i,0) = M^{i+\ell}$, and 
      \[ \overline{f(K^i)} = \dual{f}(\overline{K}^i) = g(\overline{K}^i, 0) = g(M^{i+\ell}) = 1. \]
      Hence $f(K^i) = 0$.
    \item $\alpha+1 \leq i \leq \ell$. Then $(K^i,1) = M^i$, and 
      \[ f(K^i) = g(K^i,1) = g(M^i) = 0. \]
    \item $\ell+1 \leq i \leq \ell+\alpha$. Then $(\overline{K}^i,0) = M^{i-\ell}$, and 
      \[ \overline{f(K^i)} = \dual{f}(\overline{K}^i) = g(\overline{K}^i, 0) = g(M^{i-\ell}) = 0. \]
      Hence $f(K^i) = 1$.
    \item $\ell+\alpha+1 \leq i \leq 2\ell$. Then $(K^i,1) = M^i$, and 
      \[ f(K^i) = g(K^i,1) = g(M^i) = 1. \]
  \end{itemize}
  Thus we have
  \[
    f(K_1,\dots,K_n) = \begin{pmatrix} f(K^1) \\ \vdots \\ f(K^\ell) \\ f(K^{\ell+1}) \\ \vdots \\ f(K^{2\ell}) \end{pmatrix}
      = \begin{pmatrix} 0 \\ \vdots \\ 0 \\ 1 \\ \vdots \\1 \end{pmatrix} \notin S(B_\ell),
  \]
  in contradiction to $f \in \Pol B_\ell$. We conclude that $g \in \Pol B_\ell$.
\end{proof}

\begin{corollary}
\label{cor:GS}
For any Boolean function $f \colon \IB^n \to \IB$, $G_S(f) \in S$ and for all $\ell \geq 2$, $f \in \Pol B_\ell$ if and only if $G_S(f) \in \Pol B_\ell$.
\end{corollary}

\begin{proof}
This brings together Lemmas~\ref{lemma:anytoS}, \ref{lemma:ST:fNotinBlimpliesgNotinBl} and \ref{lemma:ST:fInBlImpliesgInBl}.
\end{proof}


\subsection{Construction of $G_{M_c}(f)$ and $G_{SM}(f)$}
\label{subsec:SMT}


Let $f \colon \IB^n \to \IB$. We define the Boolean function
$G_{M_c}(f) \colon \IB^{2n} \to \IB$ by the following rules
\begin{itemize}
\item If $w(\xx) < n$, then $G_{M_c}(f)(\xx) := 0$.
\item If $w(\xx) > n$, then $G_{M_c}(f)(\xx) := 1$.
\item If $\xx = (\aa, \overline{\aa})$ for some $\aa \in \IB^n$, then $G_{M_c}(f)(\xx) := f(\aa)$.
\item If $w(\xx) = n$ and there exists $i \in [n]$ such that $x_i = x_{n+i}$ and $x_j \neq x_{n+j}$ for all $j < i$, then $G_{M_c}(f)(\xx) := x_i$.
\end{itemize}
%
It is easy to verify that the function $G_{M_c}(f)$ is defined on every tuple $\xx \in \IB^{2n}$.

\begin{lemma}
\label{lemma:StoSM}
Let $f \colon \IB^n \to \IB$.
\begin{enumerate}[label=\rm (\roman*)]
\item\label{lemma:StoSM:item1}
$G_{M_c}(f) \in M_c$, i.e., $G_{M_c}(f)$ is monotone and constant-preserving.
\item\label{lemma:StoSM:item2}
If $f$ is self-dual, then $G_{M_c}(f)$ is self-dual.
\end{enumerate}
\end{lemma}

\begin{proof}
Let $g := G_{M_c}(f)$.

\begin{asparaenum}[(i)]
\item
  Let $\xx,\yy \in \IB^{2n}$ with $\xx < \yy$. Then $w(\xx) < w(\yy)$ and one of the following cases applies: 
$w(\xx) < n$ or
$w(\yy) > n$.
In the former case, we have $g(\xx) = 0 \leq g(\yy)$;
in the latter case, we have $g(\xx) \leq 1 = g(\yy)$.
We conclude that $g$ is monotone.

Since $w(\00) = 0 < n$ and $w(\11) = 1 > n$, it holds that $f(\00) = 0$ and $f(\11) = 1$, i.e., $f$ preserves both constants.

\item
  Assume that $f$ is self-dual.
  Let $\xx \in \IB^{2n}$.
   
  If $w(\xx) > n$ then $w(\overline{\xx}) < n$, and thus $(g(\xx),g(\overline{\xx})) = (1,0)$. Similarly,
  if $w(\xx) < n$ then $w(\overline{\xx}) > n$, and thus $(g(\xx),g(\overline{\xx})) = (0,1)$.

  If $\xx = (\aa, \overline{\aa})$ for some $\aa \in \IB^n$, then 
  $(g(\xx), g(\overline{\xx})) = (f(\aa), f(\overline{\aa})) \in \{ (0,1),(1,0) \}$ since $f$ is self-dual.

  Otherwise, there is some $i \in [m]$ with $x_i = x_{m+i}$ and $x_j \neq x_{m+j}$ for all $j < i$.
  This holds also for the negation of $\xx$, and thus 
  $(g(\xx), g(\overline{\xx})) \in \{ (0,1),(1,0) \}$.

  We conclude that $g$ is self-dual.
\qedhere
\end{asparaenum}
\end{proof}

\begin{lemma} \label{lemma:SM:fnotinBlimpliesgnotinBl}
Let $f \colon \IB^n \to \IB$. If $f \notin \Pol B_\ell$ for some $\ell \geq 2$, then $G_{M_c}(f) \notin \Pol B_\ell$.
\end{lemma}

\begin{proof}
  Let $f \notin \Pol B_\ell$ and $g := G_{M_c}(f)$. 
  Then there are $\yy_1,\dots,\yy_n \in R(B_\ell)$ with $f(\yy_1,\dots,\yy_n) \notin S(B_\ell)$.
  Also $\overline{\yy}_1,\dots,\overline{\yy}_n \in R(B_\ell)$ and thus
  \[ g(\yy_1,\dots,\yy_n,\overline{\yy}_1,\dots,\overline{\yy}_n) = f(\yy_1,\dots,\yy_n) \notin S(B_\ell). \]
  Therefore $g \notin \Pol B_\ell$.
\end{proof}

\begin{lemma} \label{lemma:SM:finBlimpliesginBl}
  Let $f \colon \IB^n \to \IB$ with $f \in \Pol B_\ell$ for some $\ell \geq 2$.
  Then $G_{M_c}(f) \in \Pol B_\ell$.
\end{lemma}

\begin{proof}
Let $g := G_{M_c}(f)$.
Suppose, on the contrary, that $g \notin \Pol B_\ell$. Then there is some matrix $M$ given by
  \[
    M = \begin{pmatrix}
      m_1^1 & m_2^1 & \dots & m_{2n}^1 \\
      m_1^2 & m_2^2 & \dots & m_{2n}^2 \\
      \vdots & \vdots &  & \vdots \\
      m_1^{2\ell} & m_2^{2\ell} & \dots & m_{2n}^{2\ell}
    \end{pmatrix} = \begin{pmatrix} 
      M^1 \\
      M^2 \\
      \vdots \\
      M^{2\ell}
    \end{pmatrix}
    = (M_1,M_2,\dots,M_{2n}),
  \]
  i.e., $M^1,\dots,M^{2\ell} \in \IB^{2n}$ are the rows of $M$, 
  and $M_1,\dots,M_{2n} \in \IB^{2\ell}$ are the columns of $M$, such that
  \begin{itemize}
    \item $M_1,\dots,M_{2n} \in R(B_\ell)$, and
    \item $\zz := g(M_1,\dots,M_{2n}) := \begin{pmatrix} g(M^1) \\ \vdots \\ g(M^{2\ell}) \end{pmatrix} 
        \notin S(B_\ell)$.
  \end{itemize}
  Thus $\zz \in \{ (\underbrace{0,\dots,0}_\ell,\underbrace{1,\dots,1}_\ell), 
  (\underbrace{1,\dots,1}_\ell,\underbrace{0,\dots,0}_\ell)\}$.
  As $B_\ell$ is invariant under swapping the first $\ell$ coordinates with the last $\ell$ coordinates, we can
  assume that $\zz = (\underbrace{0,\dots,0}_\ell,\underbrace{1,\dots,1}_\ell)$.

  We have the following possibilities
  for $M^i$ with $1 \leq i \leq 2\ell$:
  \begin{enumerate}[label=(\roman*)]
    \item \label{enum:finQlimpliesginQl:smallweight}
      $w(M^i) \neq n$;
    \item \label{enum:finQlimpliesginQl:equalpairs}
      $w(M^i) = n$ and there is some $b \in [n]$ with $m^i_b = m^i_{n+b}$;
    \item \label{enum:finQlimpliesginQl:fromf}
      $w(M^i) = n$ and $m^i_b \neq m^i_{n+b}$ for all $b \in [n]$, i.e., there is some $\aa_i \in \IB^n$ with
      $M^i = (\aa_i,\overline{\aa_i})$.
  \end{enumerate}

  We show that case \ref{enum:finQlimpliesginQl:smallweight} cannot happen, since the weight of each row $M^i$ of $M$ is exactly $n$.
  Since $g(M^i) = 0$ for $1 \leq i \leq \ell$, we have $w(M^i) \leq n$ for $1 \leq i \leq \ell$. 
  Similarly, we have $w(M^i) \geq n$ for $\ell+1 \leq i \leq 2\ell$.
  Thus $\sum_{i=1}^\ell w(M^i) \leq n\ell$ and $\sum_{i=\ell+1}^{2\ell} w(M^i) \geq n\ell$.
  Because $M_j \in R(B_\ell)$ for $1 \leq j \leq 2n$, we get
  \[
    \sum_{i=1}^\ell w(M^i) = \sum_{i=1}^\ell \sum_{j=1}^{2n} m_j^i 
     = \sum_{j=1}^{2n} \sum_{i=1}^\ell m_j^i 
     = \sum_{j=1}^{2n} \sum_{i=\ell+1}^{2\ell} m_j^i 
     = \sum_{i=\ell+1}^{2\ell} \sum_{j=1}^{2n} m_j^i 
     = \sum_{i=\ell+1}^{2\ell} w(M^i)
    \]
  Therefore $\sum_{i=1}^\ell w(M^i) = \sum_{i=\ell+1}^{2\ell} w(M^i) = n\ell$, and
  $w(M^i) = n$ for $1 \leq i \leq 2\ell$.
  Thus the case \ref{enum:finQlimpliesginQl:smallweight} cannot happen for $M^i$.
  
  We will show that case \ref{enum:finQlimpliesginQl:equalpairs} is also not possible.
  Suppose, on the contrary, that there is some $i \in [2\ell]$ and some $b \in [n]$ such that
  $m^i_b = m^i_{n+b}$, and $m^i_a \neq m^i_{n+a}$ for all $a < b$. We can assume
  that $b$ is the smallest number with this property.

  Now we consider the weights of $M_b$ and $M_{n+b}$.
  Because $b$ is minimal, we have that $m^{i'}_a \neq m^{i'}_{n+a}$ for all $a < b$.
  Thus we have $(m^{i'}_b, m^{i'}_{n+b}) \in \{ (0,0),(0,1),(1,0) \}$
  for $1 \leq i' \leq \ell$, and $(m^{i'}_b, m^{i'}_{n+b}) \in \{ (0,1),(1,0),(1,1) \}$ for $\ell+1 \leq i' \leq 2\ell$. 
  Then 
  \begin{align*}
    \sum_{i' = 1}^{\ell} (m^{i'}_b + m^{i'}_{n+b}) & 
    \leq n, \\
    \sum_{i' = \ell+1}^{2\ell} (m^{i'}_b + m^{i'}_{n+b}) &  
    \geq n,
  \end{align*}
and at least one of these inequalities holds strictly.
  This implies that one of the following holds:
  \begin{align*}
    \sum_{i' = 1}^\ell m^{i'}_b    & < \sum_{i' = \ell+1}^{2\ell} m^{i'}_b         \quad \text{or} \\
    \sum_{i' = 1}^\ell m^{i'}_{n+b} & < \sum_{i' = \ell+1}^{2\ell} m^{i'}_{n+b}.
  \end{align*}
  This means that $M_b \notin R(B_\ell)$ or $M_{n+b} \notin R(B_\ell)$, in contradiction to the assumption.
  Thus no such $b$ exists, and case \ref{enum:finQlimpliesginQl:equalpairs} cannot happen.

  Thus case \ref{enum:finQlimpliesginQl:fromf} applies for all $M^i$, i.e.,
  $M^i = (\aa_i,\overline{\aa_i})$ for some $\aa_i \in \IB^n$ holds for all $i \in [2\ell]$.
  By the definition of $g$ and since $f \in \Pol B_\ell$, we obtain  
  \[
    \zz  = g\begin{pmatrix} M^1 \\ \vdots \\ M^{2\ell} \end{pmatrix} 
       = f\begin{pmatrix} \aa_1 \\ \vdots \\ \aa_{2\ell} \end{pmatrix}
       = f(M_1,\dots,M_n)
       \in S(B_\ell).
  \]
  But this is a contradiction to $\zz \notin S(B_\ell)$.
  Thus the matrix $M$ cannot exist, and we have $g \in \Pol B_\ell$.
\end{proof}

\begin{corollary}
\label{cor:GMc}
For any Boolean function $f \colon \IB^n \to \IB$, $G_{M_c}(f) \in M_c$ and for all $\ell \geq 2$, $f \in \Pol B_\ell$ if and only if $G_{M_c}(f) \in \Pol B_\ell$.
\end{corollary}

\begin{proof}
This brings together Lemmas~\ref{lemma:StoSM}\ref{lemma:StoSM:item1}, \ref{lemma:SM:fnotinBlimpliesgnotinBl} and \ref{lemma:SM:finBlimpliesginBl}.
\end{proof}

Let $G_{SM}(f) := G_{M_c}(G_{S}(f))$. Then we can conclude the following corollary from the preceding lemmas.

\begin{corollary}
\label{cor:GSM}
For any Boolean function $f \colon \IB^n \to \IB$, $G_{SM}(f) \in SM$ and for all $\ell \geq 2$, $f \in \Pol B_\ell$ if and only if $G_{SM}(f) \in \Pol B_\ell$.
\end{corollary}


\begin{proof}
By Corollary~\ref{cor:GS}, we have $G_{S}(f) \in S$, and by Lemma~\ref{lemma:StoSM}, we get $G_{SM}(f) = G_{M_c}(G_S(f)) \in SM$.

By Corollary~\ref{cor:GS}, the condition $f \in \Pol B_\ell$ is equivalent to $G_S(f) \in \Pol B_\ell$, which is in turn equivalent to $G_{SM}(f) = G_{M_c}(G_S(f)) \in \Pol B_\ell$ by Corollary~\ref{cor:GMc}.
\end{proof}

\subsection{Construction of $G_{U_\infty}(f)$, $G_{M_cU_\infty}(f)$ and $G_{M_cW_\infty}(f)$}
\label{subsec:MCT}


Let $f \colon \IB^n \to \IB$. Define $G_{U_\infty}(f) \colon \IB^{n+1} \to \IB$ by
\[
G_{U_\infty}(f)(x_1, \dots, x_{n+1}) = x_{n+1} \wedge f(x_1, \dots, x_n).
\]

\begin{lemma}
\label{lemma:SMtoMcUinfty}
Let $f \colon \IB^n \to \IB$.
\begin{enumerate}[label=\rm (\roman*)]
\item\label{lemma:SMtoMcUinfty:item1}
$G_{U_\infty}(f) \in U_\infty$.
\item\label{lemma:SMtoMcUinfty:item2}
If $f$ is monotone, then $G_{U_\infty}(f)$ is monotone.
\item\label{lemma:SMtoMcUinfty:item3}
If $f(\11) = 1$, then $G_{U_\infty}(f)$ preserves both constants.
\item\label{lemma:SMtoMcUinfty:item4}
If $f \in M_c$, then $G_{U_\infty}(f) \in M_cU_\infty$.
\end{enumerate}
\end{lemma}

\begin{proof}
Let $g := G_{U_\infty}(f)$.

\ref{lemma:SMtoMcUinfty:item1}
By the definition of $g$ we have that if $g(x_1,\dots,x_{n+1}) = 1$ then $x_{n+1} = 1$. Thus $g \in U_\infty$.

\ref{lemma:SMtoMcUinfty:item2}
Let $\xx, \yy \in \IB^{n+1}$, and assume that $\xx < \yy$. If $x_{n+1} = 0$, then $g(\xx) = 0 \leq g(\yy)$. If $x_{n+1} = 1$, then also $y_{n+1} = 1$, and since $f$ is monotone, we have
\[
g(\xx) = f(x_1, \dots, x_n) \leq f(y_1, \dots, y_n) = g(\yy).
\]
We conclude that $g$ is monotone. 

\ref{lemma:SMtoMcUinfty:item3}
By the definition of $g$, we have $g(\00) = 0$. Furthermore, if $f(\11) = 1$, then we have $g(\11) = f(\11) = 1$.

\ref{lemma:SMtoMcUinfty:item4}
Follows immediately from the previous items.
\end{proof}

\begin{lemma} \label{lemma:McUinfty:fnotinBlimpliesgnotinBl}
Let $f \colon \IB^n \to \IB$. If $f \notin \Pol B_\ell$ for some $\ell \geq 2$, then $G_{U_\infty}(f) \notin \Pol B_\ell$.
\end{lemma}

\begin{proof}
  The proof is exactly the same as the proof of Lemma~\ref{lemma:ST:fNotinBlimpliesgNotinBl}.
\end{proof}

\begin{lemma} \label{lemma:McUinfty:finBlimpliesginBl}
Let $f \colon \IB^n \to \IB$. If $f \in \Pol B_\ell$ for some $\ell \geq 2$, then $G_{U_\infty}(f) \in \Pol B_\ell$.
\end{lemma}

\begin{proof}
  Let $g := G_{U_\infty}(f)$.

  Suppose, on the contrary, that $g \notin \Pol B_\ell$. Then there is some matrix $M$ given by
  \[
    M = \begin{pmatrix}
      m_1^1 & m_2^1 & \dots & m_{n+1}^1 \\
      m_1^2 & m_2^2 & \dots & m_{n+1}^2 \\
      \vdots & \vdots &  & \vdots \\
      m_1^{2\ell} & m_2^{2\ell} & \dots & m_{n+1}^{2\ell}
    \end{pmatrix} = \begin{pmatrix} 
      M^1 \\
      M^2 \\
      \vdots \\
      M^{2\ell}
    \end{pmatrix}
    = (M_1,M_2,\dots,M_{n+1}),
  \]
  i.e., $M^1,\dots,M^{2\ell} \in \IB^{n+1}$ are the rows of $M$, 
  and $M_1,\dots,M_{n+1} \in \IB^{2\ell}$ are the columns of $M$, such that
  \begin{itemize}
    \item $M_1,\dots,M_{n+1} \in R(B_\ell)$, and
    \item $\zz := g(M_1,\dots,M_{n+1}) := \begin{pmatrix} g(M^1) \\ \vdots \\ g(M^{2\ell}) \end{pmatrix} 
        \notin S(B_\ell)$.
  \end{itemize}
  Thus $\zz \in \{ (\underbrace{0,\dots,0}_\ell,\underbrace{1,\dots,1}_\ell), 
  (\underbrace{1,\dots,1}_\ell,\underbrace{0,\dots,0}_\ell)\}$.
  As $B_\ell$ is invariant under swapping the first $\ell$ coordinates with the last $\ell$ coordinates, we can
  assume that $\zz = (\underbrace{0,\dots,0}_\ell,\underbrace{1,\dots,1}_\ell)$.

  We now look at the last column $M_{n+1}$ of $M$. Since $\sum_{i=1}^\ell m_{n+1}^i = \sum_{i=\ell+1}^{2\ell} m_{n+1}^i$,
  and since $B_\ell$ is totally symmetric on the first $\ell$ rows and on the last $\ell$ rows, we can assume
  that \[
    M_{n+1} = (\underbrace{0,\dots,0}_\alpha, \underbrace{1,\dots,1}_\beta,
               \underbrace{0,\dots,0}_\alpha, \underbrace{1,\dots,1}_\beta)
  \]
  holds for some $\alpha,\beta \geq 0$ with $\alpha+\beta = \ell$.
  
  If $\alpha > 0$ then $g(M^{\ell+1}) = g(m_1^{\ell+1},\dots,m_n^{\ell+1},0) = 0 \land f(m_1^{\ell+1},\dots,m_n^{\ell+1}) = 0$,
  in contradiction to $g(M^{\ell+1}) = 1$. Thus $\alpha = 0$, and $M_{n+1} = \11$.
  But then $f(M_1,\dots,M_n) = g(M_1,\dots,M_n,\11) = \zz \notin S(B_\ell)$, which implies
  that $f \notin \Pol B_\ell$. This contradicts the assumption $f \in \Pol B_\ell$, and we conclude that
  $g \in \Pol B_\ell$.
\end{proof}

\begin{corollary}
\label{cor:GUinfty}
For any Boolean function $f \colon \IB^n \to \IB$, $G_{U_\infty}(f) \in U_\infty$ and for all $\ell \geq 2$, $f \in \Pol B_\ell$ if and only if $G_{U_\infty}(f) \in \Pol B_\ell$.
\end{corollary}

\begin{proof}
This brings together Lemmas~\ref{lemma:SMtoMcUinfty}\ref{lemma:SMtoMcUinfty:item1}, \ref{lemma:McUinfty:fnotinBlimpliesgnotinBl} and \ref{lemma:McUinfty:finBlimpliesginBl}.
\end{proof}

Let $G_{M_cU_\infty}(f) := G_{U_\infty}(G_{M_c}(f))$ and $G_{M_cW_\infty}(f) := \dual{G_{M_cU_\infty}(f)}$.

\begin{corollary}
\label{cor:GMcUinfty}
For any Boolean function $f \colon \IB^n \to \IB$, $G_{M_cU_\infty}(f) \in M_cU_\infty$ and for all $\ell \geq 2$, $f \in \Pol B_\ell$ if and only if $G_{M_cU_\infty}(f) \in \Pol B_\ell$.
\end{corollary}

\begin{proof}
By Corollary~\ref{cor:GMc}, we have $G_{M_c}(f) \in M_c$, and by Lemma~\ref{lemma:SMtoMcUinfty}, we get $G_{M_cU_\infty}(f) = G_{U_\infty}(G_{M_c}(f)) \in M_cU_\infty$.

By Corollary~\ref{cor:GMc}, the condition $f \in \Pol B_\ell$ is equivalent to $G_{M_c}(f) \in \Pol B_\ell$, which in turn is equivalent to $G_{M_cU_\infty}(f) = G_{U_\infty}(G_{M_c}(f)) \in \Pol B_\ell$ by Corollary~\ref{cor:GUinfty}.
\end{proof}

Let $f \colon \IB^n \to \IB$. We define the functions $\overline{f} \colon \IB^n \to \IB$ and $f \colon \IB^n \to \IB$, for $\uu \in \IB^n$, as
\begin{align*}
\overline{f}(\aa) &= \overline{f(\aa)}, \\
f^\uu(\aa) &= f(\aa \oplus \uu).
\end{align*}
Note that $\dual{f} = \overline{f^\11}$, where $\11 := (1, \dots, 1) \in \IB^n$.

\begin{lemma}
\label{lem:dualBl}
Let $f \colon \IB^n \to \IB$, and let $\ell \geq 2$. The following are equivalent:
\begin{enumerate}[label=\rm (\roman*)]
\item $f \in \Pol B_\ell$,
\item $f^\uu \in \Pol B_\ell$ for any $\uu \in \IB^n$,
\item $\overline{f} \in \Pol B_\ell$,
\item $\dual{f} \in \Pol B_\ell$.
\end{enumerate}
\end{lemma}

\begin{proof}
$\text{(i)} \iff \text{(ii)}$:
Let $\aa^1, \dots \aa^n \in R(B_\ell)$. Since $R(B_\ell)$ is invariant under taking negations of its members, we also have $\overline{\aa^1}, \dots \overline{\aa^n} \in R(B_\ell)$. Let $\uu \in \IB^n$, and let $\bb^i := \aa^i$ if $u_i = 0$ and $\bb^i := \overline{\aa^i}$ if $u_i = 1$, for $i \in \nset{n}$. If $f \in \Pol B_\ell$, then
\[
f^\uu(\aa^1, \dots, \aa^n) = f(\bb^1, \dots, \bb^n) \in S(B_\ell);
\]
hence $f^\uu \in \Pol B_\ell$. The converse implication holds, since $(f^\uu)^\uu = f$.

$\text{(i)} \iff \text{(iii)}$:
Assume that $f \in \Pol B_\ell$, and let $\aa^1, \dots \aa^n \in R(B_\ell)$. Then $f(\aa^1, \dots, \aa^n) \in S(B_\ell)$. Since $S(B_\ell)$ is invariant under taking negations of its members, we have
\[
\overline{f}(\aa^1, \dots, \aa^n) 
= \overline{f(\aa^1, \dots, \aa^n)}
\in S(B_\ell);
\]
hence $\overline{f} \in \Pol B_\ell$.
The converse implication holds, since $\overline{\overline{f}} = f$.

$\text{(i)} \iff \text{(iv)}$:
This follows immediately from the equivalence of (i), (ii) and (iii), because
$\dual{f} = \overline{f^\11}$.
\end{proof}


\begin{corollary}
\label{cor:GMcWinfty}
For any Boolean function $f \colon \IB^n \to \IB$, $G_{M_cW_\infty}(f) \in M_cW_\infty$ and for all $\ell \geq 2$, $f \in \Pol B_\ell$ if and only if $G_{M_cW_\infty}(f) \in \Pol B_\ell$.
\end{corollary}

\begin{proof}
Since $M_cW_\infty = \{\dual{f} : f \in M_cU_\infty\}$, the claim follows from Lemma~\ref{lem:dualBl} and Corollary~\ref{cor:GMcUinfty}.
\end{proof}




\section{Simple games and magic squares revisited}
\label{sec:magic}

In their proof of the existence of $k$-asummable functions that are not $(k+1)$-asummable (see Theorem~\ref{theorem:strictAkInclusions}), Taylor and Zwicker constructed a certain family of functions \cite{Taylor}. We recall their construction here, and then we will refine Theorem~\ref{theorem:strictAkInclusions} and determine how the sets $\Pol B_n$ are related to each other. We will also show that Taylor and Zwicker's functions actually constitute an antichain of minimally non-threshold functions.

Fix an integer $k \geq 3$.
For $p, q \in \nset{k}$, define the $k \times k$ matrix $A^{p,q} = (a_{i,j})$ as follows:
\[
a_{i,j} =
\begin{cases}
k - 1, & \text{if $(i,j) = (p,q)$,} \\
1, & \text{if $i \neq p$ and $j \neq q$,} \\
0, & \text{otherwise.}
\end{cases}
\]
For example, if $k = 4$, then $A^{2,3} = \begin{pmatrix} 1 & 1 & 0 & 1 \\ 0 & 0 & 3 & 0 \\ 1 & 1 & 0 & 1 \\ 1 & 1 & 0 & 1 \end{pmatrix}$.
Let $B$ be the $k \times k$ matrix all of whose entries are equal to $k - 1$.

Let $S$ be a subset of $\nset{k} \times \nset{k}$.
We refer to $S$ as the $i$-th \emph{row} if $S = \{(i, j) : j \in \nset{k}\}$, and we refer to $S$ as the $j$-th \emph{column} if $S = \{(i, j) : i \in \nset{k}\}$.

\begin{lemma}
\label{lem:roworcolumn}
Let $S \subseteq \nset{k} \times \nset{k}$. Then $\sum_{(p,q) \in S} A^{p,q} = B$ if and only if $S$ is a row or a column.
\end{lemma}

\begin{proof}
It is clear that if $S$ is a row or a column, then $\sum_{(p,q) \in S} A^{p,q} = B$.

Assume then that $\sum_{(p,q) \in S} A^{p,q} = B$. Clearly $S$ is nonempty, so choose an element $(p,q)$ of $S$; clearly $S$ contains another element $(p',q')$. If $p \neq p'$ and $q \neq q'$, then the entry on row $p$ column $q$ in the sum $\sum_{(p,q) \in S} A^{p,q}$ is at least $k$; hence the sum cannot be equal to $B$. Thus either $p = p'$ or $q = q'$. It is easy to see that in the former case, all remaining entries of $S$ must be on the $p$-th row, and all elements of the $p$-th row must be in $S$; in the latter case, all remaining entries of $S$ must be on the $q$-th column, and all elements of the $q$-th column must be in $S$. We conclude that $S$ is either a row or a column.
\end{proof}

We define a function $\phi \colon \nset{R}^{k \times k} \to \IN$ that maps each $k \times k$ matrix with entries in $\nset{R}$ to an integer, where $R$ is a sufficiently large integer that will be specified below.
The function $\phi$ is defined as follows: for a matrix $M$, read the entries of $M$ from left to right and from top to bottom; the resulting string is the representation of $\phi(M)$ in base $R$.
For $p, q \in \nset{k}$, denote $w^{p,q} := \phi(A^{p,q})$ and $t := \phi(B)$.
For example if $k = 4$, then $w^{2,3} = 1101003011011101_R$ and $t = 3333333333333333_R$.
We must choose $R$ in such a way that when we add these numbers to form the sum $\sum_{(p,q) \in S} w^{p,q}$ for any $S \subseteq \nset{k} \times \nset{k}$, no carry will occur. Thus, the number $(k-1)^2 + (k-1) + 1 = k^2 - k + 1$, or anything larger, would be fine.

It is easy to see that the function $\phi$ has the following preservation property:
for any $S \subseteq \nset{k} \times \nset{k}$,
$\phi(\sum_{(p,q) \in S} A^{p,q}) = \sum_{(p,q) \in S} \phi(A^{p,q})$.
It thus follows from Lemma~\ref{lem:roworcolumn} that for all $S \subseteq \nset{k} \times \nset{k}$, it holds that $\sum_{(p,q) \in S} w^{p,q} = t$ if and only if $S$ is a row or a column.

Fix a bijection $\beta \colon \nset{k} \times \nset{k} \to \nset{k^2}$.
The \emph{characteristic tuple} of a subset $S$ of $\nset{k} \times \nset{k}$ is the tuple $\ee_S \in \IB^{k^2}$, whose $i$-th entry is $1$ if $i = \beta(p,q)$ for some $(p,q) \in S$ and $0$ otherwise. With no risk of confusion, we will refer to the characteristic tuples of rows and columns also as \emph{rows} and \emph{columns,} respectively.

Let $\ww = (w^{\beta^{-1}(1)}, \dots, w^{\beta^{-1}(k^2)})$.

For any $n$-tuples $\aa, \bb \in \IR^n$, the \emph{dot product} is defined as
\[
\aa \cdot \bb = \sum_{i=1}^n a_i b_i.
\]

Taylor and Zwicker's function $f_k \colon \IB^{k^2} \to \IB$ is defined by the following rule:
$f_k(\xx) = 1$ if and only if $\xx \cdot \ww > t$ or $\xx$ is a row.

Note that for all $\xx \in \IB^{k^2}$, $\xx \cdot \ww = t$ if and only if $\xx$ is a row or a column.

\begin{lemma}
\label{lem:fkBl}
Let $k \geq 3$ and $\ell \geq 2$. Then $f_k$ preserves $B_\ell$ if and only if $k$ is not a divisor of $\ell$.
\end{lemma}

\begin{proof}
If $\ell = mk$ for some integer $m$,
then let $\aa^1, \dots, \aa^\ell$ comprise $m$ occurrences of each column,
and let $\bb^1, \dots, \bb^\ell$ comprise $m$ occurrences of each row.
Then, the $\aa^i$ are false points of $f_k$ and the $\bb^i$ are true points, and $\aa^1 + \dots + \aa^\ell = (m, \dots, m) = \bb^1 + \dots + \bb^\ell$. Thus $f_k$ is not $\ell$-asummable. Lemma~\ref{lem:PolBl} implies that $f_k$ does not preserve $B_\ell$.

Assume then that $k$ is not a divisor of $\ell$.
Suppose, on the contrary, that $f_k$ does not preserve $B_\ell$.
By Lemma~\ref{lem:PolBl}, there exist $\aa^1, \dots, \aa^\ell \in f^{-1}(0)$ and $\bb^1, \dots, \bb^\ell \in f^{-1}(1)$ such that $\aa^1 + \dots + \aa^\ell = \bb^1 + \dots + \bb^\ell$.
Since $\xx \cdot \ww \leq t$ for any false point $\xx$ of $f_k$, and $\xx \cdot \ww \geq t$ for any true point $\xx$, we have
\[
\sum_{i = 1}^\ell \aa^i \cdot \ww \leq \ell t
\qquad
\text{and}
\qquad
\sum_{i = 1}^\ell \bb^i \cdot \ww \geq \ell t.
\]
On the other hand, since $\aa^1 + \dots + \aa^\ell = \bb^1 + \dots + \bb^\ell$, we have
\[
\sum_{i = 1}^\ell \aa^i \cdot \ww
= (\aa^1 + \dots + \aa^\ell) \cdot \ww
= (\bb^1 + \dots + \bb^\ell) \cdot \ww
= \sum_{i = 1}^\ell \bb^i \cdot \ww.
\]
Consequently, $\aa^i \cdot \ww = t$ and $\bb^i \cdot \ww = t$ for all $i \in \nset{\ell}$, and we conclude that each $\aa^i$ is a column and each $\bb^i$ is a row.
Since $k$ is not a divisor of $\ell$, there necessarily exist two columns that have a different number of occurrences among $\aa^1, \dots, \aa^\ell$. Then $\phi^{-1}(\aa^1 + \dots + \aa^n)$ is a matrix that is constant along each column, but there are two columns with distinct values.
This contradicts the fact that the matrix $\phi^{-1}(\bb^1 + \dots + \bb^n)$ is constant along each row.
This completes the proof, and we conclude that $f_k$ preserves $B_\ell$.
\end{proof}

\begin{lemma}
\label{lem:mod2sum}
The modulo-$2$ addition operation ${\oplus}$ preserves $B_\ell$ if and only if $\ell$ is odd.
\end{lemma}

\begin{proof}
The false points of $\oplus$ are $(0,0)$ and $(1,1)$, while the true points are $(0,1)$ and $(1,0)$.
Hence the sum of any $\ell$ false points is of the form $(m,m)$ for some $m$ with $0 \leq m \leq \ell$.
The sum of any $\ell$ true points is of the form $(m, \ell - m)$ for some $m$ with $0 \leq m \leq \ell$.

If $\ell$ is odd, then $m \neq \ell - m$ for any $m$. It follows that $\aa^1 + \dots + \aa^\ell \neq \bb^1 + \dots + \bb^\ell$ for any false points $\aa^1, \dots, \aa^\ell$ and any true points $\bb^1, \dots, \bb^\ell$. By Lemma~\ref{lem:PolBl}, $\oplus$ preserves $B_\ell$.

If $\ell$ is even, say $\ell = 2k$, then
\begin{multline*}
\underbrace{(0,0) + \dots + (0,0)}_k + \underbrace{(1,1) + \dots + (1,1)}_k = \\
\underbrace{(0,1) + \dots + (0,1)}_k + \underbrace{(1,0) + \dots + (1,0)}_k.
\end{multline*}
By Lemma~\ref{lem:PolBl}, $\oplus$ does not preserve $B_\ell$.
\end{proof}

\begin{proposition}
Let $\ell, m \geq 2$. 
Then $\Pol B_\ell \subseteq \Pol B_m$ if and only if $m$ divides $\ell$.
\end{proposition}

\begin{proof}
Assume first that $m$ does not divide $\ell$.
If $m \neq 2$, then by Lemma~\ref{lem:fkBl}, $f_m \in \Pol B_\ell$ but $f_m \notin \Pol B_m$.
If $m = 2$, then by Lemma~\ref{lem:mod2sum}, ${\oplus} \in \Pol B_\ell$ but ${\oplus} \notin \Pol B_m$.
In either case, we conclude that $\Pol B_\ell \not\subseteq \Pol B_m$.

Assume then that $\ell = km$ for some integer $k$.
Let $f \in \Pol B_\ell$.
Let $\aa^1, \dots, \aa^n \in R(B_m)$.
For each $i \in \{1, \dots, n\}$, define the tuple $\bb^i \in \IB^\ell$ as
\[
\bb^i = (\underbrace{a^i_1, \dots, a^i_1}_k, \dots, \underbrace{a^i_m, \dots, a^i_m}_k, \underbrace{a^i_{m + 1}, \dots, a^i_{m + 1}}_k, \dots, \underbrace{a^i_{2m}, \dots, a^i_{2m}}_{k}).
\]
It is clear that $\bb^i \in R(B_\ell)$.
Let $\zz := f(\bb^1, \dots, \bb^n)$, that is,
\[
\zz = (\underbrace{f(a^1_1, \dots, a^n_1), \dots, f(a^1_1, \dots, a^n_1)}_k, \dots, \underbrace{f(a^1_{2m}, \dots, a^n_{2m}), \dots, f(a^1_{2m}, \dots, a^n_{2m})}_k).
\]
Since $f \in \Pol B_\ell$, we have $\zz \in S(R_\ell)$. Then
\[
\zz \in \IB^\ell \setminus \{(\underbrace{0, \dots, 0}_\ell, \underbrace{1, \dots, 1}_\ell), (\underbrace{1, \dots, 1}_\ell, \underbrace{0, \dots, 0}_\ell))\}.
\]
This implies that
\[
f(\aa^1, \dots, \aa^n) \in \IB^m \setminus \{(\underbrace{0, \dots, 0}_m, \underbrace{1, \dots, 1}_m), (\underbrace{1, \dots, 1}_m, \underbrace{0, \dots, 0}_m))\}.
\]
Thus $f \in \Pol B_m$, and we conclude that $\Pol B_\ell \subseteq \Pol B_m$.
\end{proof}

\begin{proposition}
\label{prop:tzantichain}
The functions $f_k$ ($k \geq 3$) are pairwise incomparable by the minor relation.
\end{proposition}

\begin{proof}
Let $m \neq n$, and consider the comparability of $f_m$ and $f_n$.
Since all variables are essential in $f_m$ and in $f_n$, and the number of essential variables cannot increase when taking minors, we have that $f_m \not\leq f_n$ whenever $m > n$.
If $m < n$, then $n$ is not a divisor of $m$ but $n$ is a divisor of itself. By Lemma~\ref{lem:fkBl}, $f_n$ preserves $B_m$ and $f_m$ does not preserve $B_m$. Since every minor of $f_n$ preserves all relational constraints $f_n$ does, we must have that $f_m \not\leq f_n$ also in this case.
\end{proof}

\begin{proposition}
\label{prop:tzmonotone}
For every $k \geq 3$, the function $f_k$ is monotone.
\end{proposition}

\begin{proof}
Let $\xx, \yy \in \IB^{k^2}$. If $\xx < \yy$, then, since each $w^{p,q}$ is positive,
$\xx \cdot \ww < \yy \cdot \ww$.
Therefore one of the following conditions holds:
$\xx \cdot \ww < t$ or $\yy \cdot \ww > t$.
In the former case, $f(\xx) = 0 \leq f(\yy)$. In the latter case, $f(\xx) \leq 1 = f(\yy)$.
\end{proof}

\begin{proposition}
\label{prop:tzminnonthr}
For every $k \geq 3$, the function $f_k$ is minimally non-threshold.
\end{proposition}

\begin{proof}
We need to show that every identification minor of $f_k$ is threshold.
Let $(p,q)$ and $(p',q')$ be distinct elements of $\nset{k} \times \nset{k}$, let $I = \{\beta(p,q), \beta(p',q')\}$, and assume without loss of generality that $\beta(p,q) < \beta(p',q')$.
We will show that $(f_k)_I$ is $\ell$-asummable for every $\ell \geq 2$ and hence threshold by Theorem~\ref{thm:thresholdasummable}.
Let $\ell \geq 2$, and let $\aa^1, \dots, \aa^\ell \in ((f_k)_I)^{-1}(0)$, $\bb^1, \dots, \bb^\ell \in ((f_k)_I)^{-1}(1)$.
Suppose, on the contrary, that $\aa^1 + \dots + \aa^\ell = \bb^1 + \dots + \bb^\ell$.
Let $\vv \in \IB^{k^2 - 1}$ be the tuple obtained from $\ww$ by replacing its $\beta(p,q)$-th entry by $w_{\beta(p,q)} + w_{\beta(p',q')}$ and deleting the $\beta(p',q')$-th entry.
(Before proceeding, we ask the reader to recall the definition of $\delta_I$ from~\eqref{eq:deltaI} in Section~\ref{susec:MinorsConstraints}.)
It clearly holds that $\xx \cdot \vv = \xx \delta_I \cdot \ww$ for all $\xx \in \IB^{k^2 - 1}$.
Therefore $((f_k)_I)(\xx) = f_k(\xx \delta_I) = 1$ if and only if $\xx \cdot \vv = \xx \delta_I \cdot \ww > t$ or $\xx \delta_I$ is a row. Note that if $\xx \delta_I$ is a row or a column, then $\xx \cdot \vv = \xx \delta_I \cdot \ww = t$.
In a similar way as we argued in the proof of Lemma~\ref{lem:fkBl}, we have
\[
\ell t
\geq \sum_{i = 1}^\ell \aa^i \cdot \vv
= (\aa^1 + \dots + \aa^\ell) \cdot \vv
= (\bb^1 + \dots + \bb^\ell) \cdot \vv
= \sum_{i = 1}^\ell \bb^i \cdot \vv
\geq \ell t.
\]
Hence $\aa^i \cdot \vv = t$ and $\bb^i \cdot \vv = t$ for all $i \in \nset{\ell}$, that is, $\aa^i \delta_I$ is a column and $\bb^i \delta_I$ is a row for all $i \in \nset{\ell}$.

Since $(p,q) \neq (p',q')$, we have $p \neq p'$ or $q \neq q'$.
If $p \neq p'$, then none of the rows $\bb^i \delta_I$ is the $p$-th row; hence $\phi^{-1}(\bb^1 \delta_I + \dots + \bb^\ell \delta_I)$ is a matrix with a row full of $0$'s, while $\phi^{-1}(\aa^1 \delta_I + \dots + \aa^\ell \delta_I)$ has no row full of $0$'s.
If $q \neq q'$, then none of the columns $\aa^i \delta_I$ is the $q$-th column; hence $\phi^{-1}(\aa^1 \delta_I + \dots + \aa^\ell \delta_I)$ is a matrix with a column full of $0$'s, while $\phi^{-1}(\bb^1 \delta_I + \dots + \bb^\ell \delta_I)$ has no column full of $0$'s.
On the other hand,
\[
\aa^1 \delta_I + \dots + \aa^\ell \delta_I
= (\aa^1 + \dots + \aa^\ell) \delta_I
= (\bb^1 + \dots + \bb^\ell) \delta_I
= \bb^1 \delta_I + \dots + \bb^\ell \delta_I.
\]
We have reached a contradiction.

We conclude that $(f_k)_I$ is $\ell$-asummable for every $\ell \geq 2$ and hence threshold.
\end{proof}

Taylor and Zwicker's functions $f_k$ constitute an infinite antichain of monotone, minimally non-threshold functions (Propositions~\ref{prop:tzantichain}, \ref{prop:tzmonotone}, \ref{prop:tzminnonthr}). It should be noted here that this antichain does not, however, characterize the set of monotone threshold functions in terms of forbidden minors, i.e., $M \cap \thr \neq \forbid(\{f_k : k \geq 3\})$. For example, there exist self-dual monotone non-threshold functions of arity $6$ (see, e.g., \cite{BioIba}), which clearly fail to have any of the $f_n$ as a minor.


\appendix
\section{Post classes}
\label{App:Post}

We provide a concise description of all clones of Boolean functions
as well as characterizing sets of relations $R$ -- or, equivalently, relational
constraints $(R,R)$ -- for some clones; the characterizion of the remaining clones
is easily derived by noting that if $C_1 = \Pol (\mathcal{Q}_1)$ and $C_2 = \Pol(\mathcal{Q}_2)$,
then $C_1 \cap C_2 = \Pol(\mathcal{Q}_1 \cup \mathcal{Q}_2)$.
We make use of notations and terminology appearing in \cite{FP} and \cite{JGK}.

\begin{asparaitem}
\item $\Omega$ denotes the clone of all Boolean functions.  It is characterized by
the empty relation.

\item $T_0$ and $T_1$ denote the clones of $0$- and $1$-preserving functions, respectively, i.e.,
\[
T_0 = \{f \in \Omega : f(0, \dots, 0) = 0\}
\quad \text{and} \quad 
T_1 = \{f \in \Omega : f(1, \dots, 1) = 1\}.
\]
They are characterized by the unary relations $\{0\}$ and $\{1\}$, respectively.
\item $T_c$ denotes the clone of constant-preserving functions, i.e.,
$T_c = T_0 \cap T_1$.

\item $M$ denotes the clone of all monotone functions, i.e.,
\[
M = \{f \in \Omega : \text{$f(\mathbf{a}) \leq f(\mathbf{b})$ whenever $\mathbf{a} \leq \mathbf{b}$}\}.
\]
It is characterized by the binary relation 
${\leq} := \{(0,0), (0,1), (1,1)\}$.

\item $M_0 = M \cap T_0$, $M_1 = M \cap T_1$, $M_c = M \cap T_c$.

\item $S$ denotes the clone of all self-dual functions, i.e., 
\[
S = \{f \in \Omega : \dual{f} = f\}.
\]
It is characterized by the binary relation $\{(0,1),(1,0)\}$.

\item $S_c = S \cap T_c$, $SM = S \cap M$.

\item $L$ denotes the clone of all linear functions, i.e.,
\[
L = \{f \in \Omega :  f = c_0 \oplus c_1 x_1 \oplus \dots \oplus c_n x_n\}.
\]
It is characterized by the quaternary relation $\{(a,b,c,d) \in \IB^4 : a \oplus b \oplus c = d\}$.

\item $L_0 = L \cap T_0$, $L_1 = L \cap T_1$, $LS = L \cap S$, $L_c = L \cap T_c$.
\end{asparaitem}

Let $a \in \{0,1\}$. A set $A \subseteq \{0,1\}^n$ is said to be \emph{$a$-separating} if there is some $i\in [n]$
such that for every $(a_1, \dotsc, a_n) \in A$ we have $a_i = a$. 
A function $f$ is said to be \emph{$a$-separating} if $f^{-1}(a)$ is $a$-separating. 
The function $f$ is said to be \emph{$a$-separating of rank $k \geq 2$} if every subset 
$A \subseteq f^{-1}(a)$ of size at most $k$ is $a$-separating.

\begin{asparaitem}
\item For $m \geq 2$, $U_m$ and $W_m$ denote the clones of all $1$- and $0$-separating functions of rank $m$, respectively.
They are characterized by the $m$-ary relations $\IB^m \setminus \{(0, \dots, 0)\}$ and $\IB^m \setminus \{(1, \dots, 1)\}$, respectively.

\item $U_\infty$ and $W_\infty$ denote the clones of all $1$- and $0$-separating functions, respectively, i.e., 
$U_\infty = \bigcap_{k \geq 2} U_k$ and $W_\infty = \bigcap_{k \geq 2} W_k$.

\item $T_cU_m = T_c \cap U_m$ and $T_cW_m = T_c \cap W_m$, for $m = 2, \dotsc, \infty$.

\item $MU_m = M \cap U_m$ and $MW_m = M \cap W_m$, for $m = 2, \dotsc, \infty$.

\item $M_cU_m = M_c \cap U_m$ and $M_cW_m = M_c \cap W_m$, for $m = 2, \dotsc, \infty$.

\item $\Lambda $ denotes the clone of all conjunctions and constants, i.e.,
\begin{multline*}
\Lambda = \{f \in \Omega : f = x_{i_1} \wedge \dotsb \wedge x_{i_n}\} \cup 
\{\mathbf{0}^{(n)}: n\geq 1\} \cup \{\mathbf{1}^{(n)}: n\geq 1\}.
\end{multline*}
It is characterized by the ternary relation 
$\{(a,b,c) : a \wedge b = c\}$.

\item $\Lambda_0 = \Lambda \cap T_0$, $\Lambda_1 = \Lambda \cap T_1$, $\Lambda_c = \Lambda \cap T_c$. 

\item $V$ denotes the clone of all disjunctions and constants, i.e.,
\begin{multline*}
V = \{f \in \Omega : f = x_{i_1} \vee \dotsb \vee x_{i_n}\}  \cup 
\{\mathbf{0}^{(n)}: n\geq 1\} \cup \{\mathbf{1}^{(n)}: n\geq 1\}.
\end{multline*}
It is characterized by the ternary relation 
$\{(a,b,c) : a \vee b = c\}$.

\item $V_0 = V \cap T_0$, $V_1 = V \cap T_1$, $V_c = V \cap T_c$.

\item $\Omega (1)$ denotes the clone of all projections, negations, and constants. 
It is characterized by the ternary relation 
$\{(a,b,c) : \text{$a = b$ or $b = c$}\}$.

\item $I^* = \Omega(1) \cap S$, $I = \Omega(1) \cap M$.

\item $I_0 = I \cap T_0$, $I_1 = I \cap T_1$.

\item $I_c$ denotes the smallest clone containing only projections, i.e., $I_c = I \cap T_c$.
\end{asparaitem}


\end{document}